\documentclass[11pt,reqno]{amsart}
\usepackage[left=33mm,right=33mm,top=30mm,bottom=32mm]{geometry}
\usepackage{mathtools,amssymb,amsthm,mathrsfs,color,float}
\usepackage{paralist}
\usepackage{stackengine}
\usepackage{centernot}
\usepackage{txfonts}
\usepackage{tabularx}
\usepackage{longtable}
\usepackage[colorlinks,
linkcolor=red,
anchorcolor=green,
citecolor=blue, 
]{hyperref}

\usepackage[T1]{fontenc}
\usepackage[utf8]{inputenc}
\usepackage{booktabs}

\usepackage{calc}
\definecolor{bleu1}{RGB}{0,57,128}
\def\bleu1{\color{bleu1}}

\usepackage{etoolbox}
\patchcmd{\section}{\normalfont}{\normalfont \bleu1}{}{}
\patchcmd{\subsection}{\normalfont}{\normalfont \bleu1}{}{}
\patchcmd{\subsubsection}{\normalfont}{\normalfont \bleu1}{}{}

\newtheorem{The}{\bleu1 Theorem}[section]
\newtheorem{Pro}{\bleu1 Proposition}[section]
\newtheorem{Lem}{\bleu1 Lemma}[section]
\newtheorem{Cor}{\bleu1 Corollary}[section]

\theoremstyle{definition}
\newtheorem{defn}{\bleu1 Definition}[section]

\newcounter{Mr}
\newtheorem{Result}[Mr]{\textbf{Main Result}}

\newcommand{\R}{\mathbb{R}}

\makeatletter
\@namedef{subjclassname@2020}{\textup{2020} Mathematics Subject Classification}
\makeatother

\title[Regularity of the fundamental solutions]{Generalized Hamiltonian gradient flow of contact type I: regularity of the fundamental solutions}
\author{Wei Cheng, Shengqing Hu \and Kaizhi Wang}
\address[Wei Cheng]{School of Mathematics, Nanjing University, Nanjing 210093, China}
\email{chengwei@nju.edu.cn}
\address[Shengqing Hu]{Faculty of Computational Mathematics and Cybernetics, Shenzhen MSU-BIT University, 518172, Shenzhen, China}
\email{hushengqing@smbu.edu.cn}
\address{School of Mathematical Sciences, Shanghai Jiao Tong University, Shanghai 200240, China}
\email{kzwang@sjtu.edu.cn}

\begin{document}
\maketitle


\begin{abstract}
This paper studies generalized Hamiltonian gradient flows for contact-type Hamilton--Jacobi equations, adopting the variational framework of Herglotz’s principle and its fundamental solution \(h_L(t,x,y,u)\). Two main results are presented. First, precise first-order sensitivity relations are derived, linking derivatives of \(h_L\) to dual arcs of minimizing trajectories. Second, quantitative second-order estimates show that \(h_L\) is locally semiconcave and, over short time, semiconvex—indeed uniformly convex in certain variables. These regularity properties follow from a detailed variational analysis of minimizing trajectories. The work establishes a foundation for intrinsic methods in analyzing singularity propagation and generalized gradient flows in contact context, with implications for weak KAM theory, optimal transport, and regularity of viscosity solutions.
\end{abstract}

\section{Introduction}

This is the first in a series of papers studying the theory of \emph{generalized Hamiltonian gradient flows} developed in \cite{CCHW2024} in the context of contact type Hamilton--Jacobi equations. Let $M$ be a smooth, connected, compact manifold without boundary, and let $TM$ and $T^{*}M$ denote its tangent and cotangent bundles, respectively. Let $H:T^{*}M\to\mathbb{R}$ be a Tonelli Hamiltonian and $L:TM\to\mathbb{R}$ the associated Tonelli Lagrangian.

Given an interval $I$ (which may be the entire real line), we say that a locally absolutely continuous curve $\gamma:I\to M$ is a \emph{maximal slope curve for the pair $(\phi,H)$}, where $\phi$ is a semiconcave function on $M$ and $H$ a Tonelli Hamiltonian on $M$, if $\gamma$ satisfies the following energy-dissipation inequality (EDI) type variational inequality:
\begin{equation}\label{eq:msc_H_phi}\tag{EDI}
    \phi(\gamma(t_{2}))-\phi(\gamma(t_{1}))\leqslant\int_{t_{1}}^{t_{2}}\Bigl\{L(\gamma(s),\dot{\gamma}(s))+H\bigl(\gamma(s),\mathbf{p}^{\#}_{\phi,H}(\gamma(s))\bigr)\Bigr\}\,ds,
    \qquad \forall t_{1},t_{2}\in I,\;t_{1}\leqslant t_{2},
\end{equation}
with equality holding for such a curve $\gamma$. Here the \emph{minimal energy selection} is defined as a Borel measurable selection from the superdifferential $D^{+}\phi(x)$:
\begin{equation*}
    \mathbf{p}^{\#}_{\phi,H}(x)=\arg\min\bigl\{H(x,p):p\in D^{+}\phi(x)\bigr\},\qquad x\in M.
\end{equation*}

By the equality condition in Fenchel-Young's inequality, $\gamma$ is a maximal slope curve for the pair $(\phi,H)$ if and only if it satisfies the following equation:
\begin{equation}\label{eq:SGC1}\tag{SC}
    \dot{\gamma}(t)=H_{p}\bigl(\gamma(t),\mathbf{p}^{\#}_{\phi,H}(\gamma(t))\bigr),\qquad \text{a.e. } t\in I.
\end{equation}
We refer to \eqref{eq:SGC1} as the \emph{generalized Hamiltonian gradient flow} associated with $(\phi,H)$, and call each solution a \emph{strict singular characteristic}. From \eqref{eq:msc_H_phi} we observe that when $\gamma$ is a backward calibrated curve, the energy dissipation term $\mathbf{p}^{\#}_{\phi,H}(\gamma(s))$ is constant and equals Ma\~n\'e's critical value $c[H]$. In the forward direction, however, this term captures a certain irreversibility phenomenon intrinsic to Hamiltonian systems. Building on an intrinsic method developed in a series of papers \cite{Cannarsa_Cheng3,Cannarsa_Cheng_Fathi2017,Cannarsa_Cheng_Fathi2021,Cannarsa_Cheng_Hong2025}, the authors proved in \cite{CCHW2024} that every solution of \eqref{eq:SGC1} propagates cut points globally whenever the initial point is a cut point.

Contact type Hamilton--Jacobi equations have attracted considerable attention in recent decades. They constitute an important class of dissipative Hamiltonian systems with numerous applications. In the present work, we focus on the Euclidean setting.

Let $L=L(x,v,r):\mathbb{R}^{n}\times\mathbb{R}^{n}\times\mathbb{R}\to\mathbb{R}$ be a $C^{2}$ function satisfying the following standing assumptions:
\begin{enumerate}[(L1)]
    \item $L(x,\cdot,r)$ is strictly convex for every $(x,r)\in\mathbb{R}^{n}\times\mathbb{R}$.
    \item There exist two superlinear functions $\overline{\theta}_{0},\theta_{0}:[0,+\infty)\to[0,+\infty)$ and constants $c_{0},c_{1}\geqslant0$ such that
    \begin{align*}
        \overline{\theta}_{0}(|v|)+c_{1}\geqslant L(x,v,0)\geqslant\theta_{0}(|v|)-c_{0},\quad 
        \forall (x,v)\in\mathbb{R}^{n}\times\mathbb{R}^{n}.
    \end{align*}
    \item There exists a constant $K>0$ such that
    \begin{align*}
        |L_{r}(x,v,r)|\leqslant K,\quad 
        \forall (x,v,r)\in\mathbb{R}^{n}\times\mathbb{R}^{n}\times\mathbb{R}.
    \end{align*}
\end{enumerate}

The associated Hamiltonian $H:\mathbb{R}^{n}\times\mathbb{R}^{n}\times\mathbb{R}\to\mathbb{R}$ is defined by
\begin{align*}
    H(x,p,r)=\sup_{v\in\mathbb{R}^{n}}\bigl\{\langle p,v\rangle-L(x,v,r)\bigr\},\quad 
    \forall (x,p,r)\in\mathbb{R}^{n}\times\mathbb{R}^{n}\times\mathbb{R}.
\end{align*}
Then $H$ is also of class $C^{2}$ and satisfies the following conditions:
\begin{enumerate}[(H1)]
    \item $H(x,\cdot,r)$ is strictly convex for every $(x,r)\in\mathbb{R}^{n}\times\mathbb{R}$.
    \item There exist two superlinear functions $\overline{\Theta}_{0},\Theta_{0}:[0,+\infty)\to[0,+\infty)$ and constants $c^{*}_{0},c^{*}_{1}\geqslant0$ such that
    \begin{align*}
        \overline{\Theta}_{0}(|v|)+c^{*}_{1}\geqslant H(x,p,0)\geqslant\Theta_{0}(|v|)-c^{*}_{0},\quad 
        \forall (x,v)\in\mathbb{R}^{n}\times\mathbb{R}^{n}.
    \end{align*}
    \item There exists a constant $K^{*}>0$ such that
    \begin{align*}
        |H_{r}(x,p,r)|\leqslant K^{*},\quad 
        \forall (x,p,r)\in\mathbb{R}^{n}\times\mathbb{R}^{n}\times\mathbb{R}.
    \end{align*}
\end{enumerate}

Let us recall the \emph{Herglotz generalized variational problem}. Fix $x,y\in\mathbb{R}^{n}$, $t>0$, and $u\in\mathbb{R}$. Define
\begin{align*}
    \Gamma^{t}_{x,y}=\bigl\{\xi\in C^{2}([0,t],\mathbb{R}^{n}): \xi(0)=x,\ \xi(t)=y\bigr\}.
\end{align*}
For a given $\xi\in\Gamma^{t}_{x,y}$, let $u_{\xi}$ be the unique $C^{2}$ solution of the ODE
\begin{equation}\label{eq:Caratheodory_ODE}
    \begin{cases}
        \dot{u}_{\xi}(s)=L\bigl(\xi(s),\dot{\xi}(s),u_{\xi}(s)\bigr), & s\in[0,t],\\[2mm]
        u_{\xi}(0)=u.
    \end{cases}
\end{equation}
We then define
\begin{equation}\label{eq:fundamental_solution}
    \begin{split}
        h_{L}(t,x,y,u) &:=\inf_{\xi\in\Gamma^{t}_{x,y}}\int^{t}_{0}L\bigl(\xi(s),\dot{\xi}(s),u_{\xi}(s)\bigr)\,ds\\
        &\;=\inf_{\xi\in\Gamma^{t}_{x,y}}\bigl\{u_{\xi}(t)-u\bigr\}.
    \end{split}
\end{equation}
Notice that if the infimum in \eqref{eq:fundamental_solution} is taken over all absolutely continuous curves connecting $x$ to $y$, it is still attained, and any minimizer $\xi$ is in fact of class $C^{2}$. For convenience, we therefore assume $\xi\in C^{2}([0,t],\mathbb{R}^{n})$ in \eqref{eq:fundamental_solution}. 

It should be noted that an alternative, implicit variational principle developed earlier by Lin Wang, Jun Yan and the third author \cite{Wang_Wang_Yan2017,Wang_Wang_Yan2019_1} has also been successfully applied to contact type Hamilton--Jacobi equations and to contact type weak KAM theory. The dynamics of contact Hamiltonian systems and their corresponding mappings, as well as issues such as the well-posedness and long-term behavior of contact-type Hamilton-Jacobi equations, are receiving increasing attention. See, for instance, \cite{Arnaud_Florio_Roos2024,Zavidovique2022,Maro_Sorrentino2017,de_Leon_Izquierdo-Lopez2024,SCdLSS2024}.

\begin{Pro}[\cite{CCWY2019,CCJWY2020}]\label{Herglotz_Lie}
	Let $L$ satisfy conditions \mbox{\rm (L1)-(L3)}. Then for fixed $x,y\in\R^n$, $t>0$ and $u\in\R$, the variational problem \eqref{eq:fundamental_solution} admits a minimizer. Moreover,
\begin{enumerate}[\rm (a)]
  \item Both $\xi$ and $u_{\xi}$ are of class $C^2$ and $\xi$ satisfies Herglotz equation
  \begin{equation}\label{eq:Herglotz}
	\begin{split}
		&\,\frac d{ds}L_v(\xi(s),\dot{\xi}(s),u_{\xi}(s))\\
	=&\,L_x(\xi(s),\dot{\xi}(s),u_{\xi}(s))+L_u(\xi(s),\dot{\xi}(s),u_{\xi}(s))L_v(\xi(s),\dot{\xi}(s),u_{\xi}(s)).
	\end{split}
  \end{equation}
       for all $s\in[0,t]$ where $u_{\xi}$ is the unique solution of \eqref{eq:Caratheodory_ODE};
  \item Let $p(s)=L_v(\xi(s),\dot{\xi}(s),u_{\xi}(s))$ be the dual arc, then $p$ is also of class $C^2$ and we conclude that $(\xi,p,u_{\xi})$ satisfies Lie equation
  \begin{equation}\label{eq:Lie}
  	\begin{cases}
  		\dot{\xi}(s)=H_p(\xi(s),p(s),u_{\xi}(s));\\
  		\dot{p}(s)=-H_x(\xi(s),p(s),u_{\xi}(s))-H_u(\xi(s),p(s),u_{\xi}(s))p(s);\\
  		\dot{u}_{\xi}(s)=p(s)\cdot\dot{\xi}(s)-H(\xi(s),p(s),u_{\xi}(s)).
  	\end{cases}
  \end{equation}
\end{enumerate}
\end{Pro}

An associated variational problem of Herglotz' type is as follows:
\begin{equation}\label{eq:fundamental_solution_2}
	\breve{h}_L(t,x,y,u):=\inf_{\xi}\int^t_0L(\xi(s),\dot{\xi}(s),w_{\xi}(s))\ ds
\end{equation}
where the infimum is taken over all $\xi\in\Gamma^t_{x,y}$ such that a terminal condition problem of Carath\'eodory equation
\begin{equation}\label{eq:caratheodory_L_2}
	\begin{cases}
		\dot{w}_{\xi}(s)=L(\xi(s),\dot{\xi}(s),w_{\xi}(s)),\quad s\in[0,t],&\\
		w_{\xi}(t)=u,&
	\end{cases}
\end{equation}
admits a (unique) solution.

We call the function $h_L(t,x,y,u)$ (resp. $\breve{h}_L(t,x,y,u)$) the {\em negative} (resp. {\em positive}) {\em type fundamental solution} for the Hamilton-Jacobi equation
\begin{equation}\label{eq:HJe_t}\tag{HJ}
D_tu(t,x)+H(x,D_xu(t,x),u(t,x))=0.
\end{equation}


Now we can formulate main results of this paper.

\begin{Result}[First order regularity]
Let $\xi\in\Gamma^t_{x,y}$ be a minimizer of \eqref{eq:fundamental_solution} and $u_{\xi}(s)$ be the unique solution of \eqref{eq:Caratheodory_ODE}.
\begin{itemize}
	\item For any $s\in[0,t]$, the sensitive relation $p_y(s):=L_v(\xi(s),\dot{\xi}(s),u_{\xi}(s))\in D^+_{y}h_L(s,x,\xi(s),u)$ holds. Moreover, in the case that $h_L(t,x,\cdot,u)$ is differentiable at $y$, we conclude that
	    \begin{align*}
		D_yh_L(t,x,y,u)=L_v(\xi(t),\dot{\xi}(t),u_{\xi}(t)).
		\end{align*}
	\item We have that $p_x(t):=-e^{\int^t_0L_u(\xi,\dot{\xi},u_{\xi})d\tau}\cdot L_v(\xi(0),\dot{\xi}(0),u_{\xi}(0))\in D^+_xh_L(t,x,y,u)$. Moreover, if $h_L(t,\cdot,y,u)$ is differentiable at $x$, then $D_xh_L(t,x,y,u)=p_x(t)$.
	\item We have that $p_u(t):=e^{\int^t_0L_u(\xi,\dot{\xi},u_{\xi})\ d\tau}-1\in D^+_uh_L(t,x,y,u)$. Moreover, if $h_L(t,x,y,\cdot)$ is differentiable at $u$, then $D_uh_L(t,x,y,u)=p_u(t)$.
	\item We have that $-e^{\int^t_0L_u(\xi,u_{\xi},\dot{\xi})\ ds}\cdot E(t)\in D^+_th_L(t,x,y,u)$. Moreover, if $h_L(\cdot,x,y,u)$ is differentiable at $t$, then $D_th_L(t,x,y,u)=-e^{\int^t_0L_u(\xi,\dot{\xi},u_{\xi})\ ds}\cdot E(t)=-H(\xi(t),p(t),u_{\xi}(t))$, where
	\begin{align*}
		E(s):=e^{-\int^s_0L_u(\xi,\dot{\xi},u_{\xi})\ d\tau}\cdot[L_v(\xi(s),\dot{\xi}(s),u_{\xi}(s))\cdot\dot{\xi}(s)-L(\xi(s),\dot{\xi}(s),u_{\xi}(s))],\qquad s\in[0,t].
	\end{align*}
\end{itemize}
\end{Result}

\begin{Result}[Second order regularity]
\hfill
\begin{itemize}
	\item \textbf{Semiconcavity}: For any $z$, $w\in\R^n$ and $h$, $k\in\R$ with $|z|,|w|\leqslant\lambda t$, $|k|\leqslant\delta$, $|h|<\frac t2$, there exist constants $C>0$ such that
	\begin{align*}
		h_L(&\,t+h,x+z,y+w,u+k)+h_L(t-h,x-z,y-w,u-k)-2h_L(t,x,y,u)\\
		\leqslant&\, \frac {C}t(|h|^2+|z|^2+|w|^2)+C|k|^2.
	\end{align*}
	\item \textbf{Semiconvexity for short time}: For any $\lambda,\delta>0$ there exists $t_{\lambda,\delta}>0$ such that for any $x\in\mathbb{R}^n$, the function $(t,y,u)\mapsto h_L(t,x,y,u)$ is semiconvex on the cone
	\begin{align*}
		S_{\lambda,\delta}=\{(t,y,u): t\in(0,t_{\lambda,\delta}], |y-x|<\lambda t,|u|<\delta\}.
	\end{align*}
	\item \textbf{Convexity in $x$ for short time}: In particular, there exist $t^\prime_{\lambda,\delta}$ and $C_2>0$ depending on $\lambda,\delta$, such that for all $t\in (0,t^\prime_{\lambda,\delta}]$ the function $h_L(t,x,\cdot,u)$ is uniformly convex on $B(x,\lambda t)$, i.e., for all $|y-x|<\lambda t$ and $|z|<\lambda t$ we have that
	\begin{align*}
		h_L(t,x,y+z,u)+h_L(t,x,y-z,u)-2h_L(t,x,y,u)
		\geqslant\,\frac{C_2}{t}|z|^2.
	\end{align*}
	\item \textbf{Convexity in $t$ for short time}: Fix $\lambda>0$. Let $x,y\in\R^n$, $u,h\in\R$ and $0<t\leqslant 1$ such that $|y-x|\leqslant\lambda t$, $|u|\leqslant \delta$ and $|h|<\frac{t}{2}$. Suppose that $\xi_\pm$ is a minimal curve for $h_L(t\pm h,x,y,u)$ and
	\begin{align*}
		\inf_{s\in [0,t+h]}|\dot{\xi}_+|^2+\inf_{s\in [0,t-h]}|\dot{\xi}_-|^2\geqslant C_1,
	\end{align*}
	where $C_1>0$ is a constant depending on $\lambda,\delta$. Then for any $x\in\mathbb{R}^n$, the function $t\rightarrow h_L(t,x,y,u)$ is uniformly convex. Thus, there exists $C_2$ such that
	\begin{align*}
		h_L(t+h,x,y,u)+h_L(t-h,x,y,u)-2h_L(t,x,y,u)\geqslant\,\frac{C_2}{t}|h|^2.
	\end{align*}
	\item \textbf{$C^{1,1}$ regularity}: Fix $\lambda>0$, $x\in\R^n$, then there exist constant $t_{\lambda}$ and $C_{\lambda}>0$ depending on $\lambda,\,x$ and $u$ such that the map $(t,y,u)\mapsto h_L(t,x,y,u)$ is locally $C^{1,1}$ on the set
	$$S=\{(t,y,u)\in \mathbb{R}\times \mathbb{R}^n\times\mathbb{R}:0<t<t_\lambda,\,|y-x|<\lambda t,\,|u|<\delta\}.$$
	Moreover, for all $(t,y,u)\in S$,
	\begin{align*}
	                 D_th_L(t,x,y,u)=&\,-e^{\int^t_0L_u(\xi,u_{\xi},\dot{\xi})\ ds}E(t)=-H(\xi(t),u_{\xi}(t),p(t)),\\
		D_yh_L(t,x,y,u)=&\,L_v(\xi(t),u_{\xi}(t),\dot{\xi}(t)),\\
		D_xh_L(t,x,y,u)=&\,-e^{\int^t_0L_u(\xi,u_{\xi},\dot{\xi})d\tau}\cdot L_v(\xi(0),u_{\xi}(0),\dot{\xi}(0)),\\
		D_uh_L(t,x,y,u)=&\,e^{\int^t_0L_u(\xi,u_{\xi},\dot{\xi})\ d\tau}-1,
	\end{align*}
	where $\xi\in\Gamma^t_{x,y}$ is the unique minimizer for $h_L(t,x,y,u)$ with $u_{\xi}$ determined by \eqref{eq:Caratheodory_ODE} and $p(\cdot)$ is the dual arc of $\xi(\cdot)$.
\end{itemize}
\end{Result}

The application of our intrinsic method to generalized Hamiltonian gradient flows of contact type relies crucially on obtaining qualitative estimates for the first- and second-order regularity of the fundamental solution, as well as on clarifying its relationship with the underlying contact Hamiltonian system. In fact, such regularity results are also essential for many other applications, including Lasry--Lions regularization (\cite{Bernard2007,Bernard2010,Bernard2012,Chen_Cheng2016}), Arnaud's theorem on the evolution of the enlarged graph of $Du$ for a semiconcave function $u$ (\cite{Arnaud2011}), SBV regularity of viscosity solutions (\cite{Bianchini_Tonon2012}), Lax--Oleinik commutators (\cite{Cannarsa_Cheng_Hong2025}) and related problems in optimal transport (\cite{Bernard_Buffoni2007a,CCSW2025}), and the global propagation of singularities for viscosity solutions (\cite{Cannarsa_Cheng3,Cannarsa_Cheng_Fathi2017,Cannarsa_Cheng_Fathi2021,CCHW2024}).

The paper is organized as follows. In Section~2, we collect the main first- and second-order regularity properties of the fundamental solution. All detailed proofs of the statements in Section~2 are provided in Section~3.

\medskip

\noindent\textbf{Acknowledgements.} Wei Cheng was partially supported by the National Natural Science Foundation of China (Grant No. 12231010). Shengqing Hu was partially supported by the National Natural Science Foundation of China (Grant No. 12201532). Kaizhi Wang was partially supported by the National Natural Science Foundation of China (Grant Nos. 12525107 and 12171315).

\section{First and second order regularity of $h_L$}
\label{sec:statement_regularity}

In this section we collect several first- and second-order regularity results for $h_{L}$. Since the arguments involve many technical details, we postpone all proofs to the next section. This presentation will help the reader grasp the main ideas of the paper without being burdened by the technicalities.

We begin by recalling basic notions of semiconcave and semiconvex functions and their corresponding super- and sub‑differentials. Let $\Omega\subset\mathbb{R}^{n}$ be a convex set. A function $u:\Omega\to\mathbb{R}$ is called \emph{semiconcave} (with linear modulus) if there exists a constant $C>0$ such that
\begin{equation}\label{eq:SCC}
    \lambda u(x)+(1-\lambda)u(y)-u\bigl(\lambda x+(1-\lambda)y\bigr)\leqslant\frac{C}{2}\,\lambda(1-\lambda)|x-y|^{2}
\end{equation}
holds for every $x,y\in\Omega$ and every $\lambda\in[0,1]$. Any constant $C$ satisfying the above inequality is called a \emph{semiconcavity constant} of $u$ on $\Omega$. A function $u:\Omega\to\mathbb{R}$ is said to be \emph{semiconvex} if $-u$ is semiconcave. When $u$ is continuous, it can be shown that $u$ is semiconcave with constant $C$ if and only if
\[
    u(x)+u(y)-2u\!\left(\frac{x+y}{2}\right)\leqslant\frac{C}{2}\,|x-y|^{2}
\]
for all $x,y\in\Omega$. We say that $u$ is \emph{locally semiconcave} (respectively, \emph{locally semiconvex}) if for each $x\in\Omega$ there exists an open ball $B(x,r)\subset\Omega$ such that $u$ is semiconcave (respectively, semiconvex) on $B(x,r)$. For a continuous function $u:\Omega\subset\mathbb{R}^{n}\to\mathbb{R}$ and any point $x\in\Omega$, the closed convex sets
\begin{align*}
    D^{-}u(x)&=\Bigl\{p\in\mathbb{R}^{n}:\liminf_{y\to x}\frac{u(y)-u(x)-\langle p,y-x\rangle}{|y-x|}\geqslant0\Bigr\},\\[2mm]
    D^{+}u(x)&=\Bigl\{p\in\mathbb{R}^{n}:\limsup_{y\to x}\frac{u(y)-u(x)-\langle p,y-x\rangle}{|y-x|}\leqslant0\Bigr\},
\end{align*}
are called, respectively, the \emph{(Dini) subdifferential} and the \emph{superdifferential} of $u$ at $x$.

\subsection{First order regularity of $h_L$ and its dynamical implication}

A preliminary inspection shows that the function $(t,x,y,u)\mapsto h_L(t,x,y,u)$ is locally semiconcave (see Theorem~\ref{semiconcavity} below). Consequently, the superdifferential $D^+h_L$ is well defined and forms a non‑empty, closed, convex subset of the ambient Euclidean space. The following first‑order properties of the fundamental solution $h_L$ will be essential for a deeper analysis of the links between Hamilton–Jacobi theory and the underlying dynamics.

\begin{Pro}\label{D_h_L}
	Let $\xi\in\Gamma^t_{x,y}$ be a minimizer of \eqref{eq:fundamental_solution} and $u_{\xi}(s)$ be the unique solution of \eqref{eq:Caratheodory_ODE}. Then, we have the following sensitive relations:
	\begin{enumerate}[\rm (a)]
		\item For any $s\in[0,t]$, we have the sensitive relation
		\begin{equation}\label{eq:p_in_superdiff}
	         p_y(s):=L_v(\xi(s),\dot{\xi}(s),u_{\xi}(s))\in D^+_{y}h_L(s,x,\xi(s),u).
	    \end{equation}
	    Moreover, in the case that $h_L(t,x,\cdot,u)$ is differentiable at $y$, we conclude that
	    \begin{align*}
		D_yh_L(t,x,y,u)=L_v(\xi(t),\dot{\xi}(t),u_{\xi}(t)).
		\end{align*}
		\item We have that
		\begin{equation}\label{eq:x_in_superdiff}
		p_x(t):=-e^{\int^t_0L_u(\xi,\dot{\xi},u_{\xi})d\tau}\cdot L_v(\xi(0),\dot{\xi}(0),u_{\xi}(0))\in D^+_xh_L(t,x,y,u).
		\end{equation}
		Moreover, if $h_L(t,\cdot,y,u)$ is differentiable at $x$, then $D_xh_L(t,x,y,u)=p_x(t)$.
		\item We have that
		\begin{equation}\label{eq:u_in_superdiff}
		p_u(t):=e^{\int^t_0L_u(\xi,\dot{\xi},u_{\xi})\ d\tau}-1\in D^+_uh_L(t,x,y,u).
		\end{equation}
		Moreover, if $h_L(t,x,y,\cdot)$ is differentiable at $u$, then $D_uh_L(t,x,y,u)=p_u(t)$.
	\end{enumerate}
\end{Pro}

\begin{defn}
	For any extremal $(\xi,u_{\xi})$ for \eqref{eq:Herglotz} on $[0,t]$, we define the function $E=E_\xi:[0,t]\to\R$ with respect to $L$ as 
\begin{align*}
	E(s):=e^{-\int^s_0L_u(\xi,\dot{\xi},u_{\xi})\ d\tau}\cdot[L_v(\xi(s),\dot{\xi}(s),u_{\xi}(s))\cdot\dot{\xi}(s)-L(\xi(s),\dot{\xi}(s),u_{\xi}(s))].
\end{align*}
For any solution $(\xi,p,u_{\xi})$ for \eqref{eq:Lie} on $[0,t]$, we also define the function $E^*=E^*_\xi:[0,t]\to\R$ with respect to $H$ as
\begin{align*}
	E^*(s):=e^{\int^s_0H_u(\xi,p,u_{\xi})\ d\tau}\cdot H(\xi(s),p(s),u_{\xi}(s)).
\end{align*}
\end{defn}

\begin{Lem}\label{energy}
Both $E$ and $E^*$ are constant on $[0,t]$.
\end{Lem}

\begin{Pro}\label{D_t}
Let $\xi\in\Gamma^t_{x,y}$ be a minimizer of \eqref{eq:fundamental_solution} and let $u_{\xi}(s)$ be the unique solution of \eqref{eq:Caratheodory_ODE}. Then $-e^{\int^t_0L_u(\xi,u_{\xi},\dot{\xi})\ ds}\cdot E(t)\in D^+_th_L(t,x,y,u)$. Moreover, if $h_L(\cdot,x,y,u)$ is differentiable at $t$, then
\begin{equation}\label{diff_of_fund_sol_t}
\,D_th_L(t,x,y,u)=-e^{\int^t_0L_u(\xi,\dot{\xi},u_{\xi})\ ds}\cdot E(t)=-H(\xi(t),p(t),u_{\xi}(t)).
\end{equation}
\end{Pro}

In the last part of this section, we discuss the relationship between $h_L$ and $\breve{h}_L$. Define $\breve{L}(x,v,r)=L(x,-v,-r)$. The following result shows that all regularity properties for the positive‑type fundamental solutions can be deduced in a similar manner from our previous results on the negative‑type fundamental solutions.

\begin{Lem}\label{eq:equiv_positive}
	Let $x,y\in\R^n$, $t>0$ and $u\in\R$. Then we have that
	\begin{align*}
		\breve{h}_L(t,x,y,u)=h_{\breve{L}}(t,y,x,-u).
	\end{align*}
	Moreover, for any $\xi\in\Gamma^t_{y,x}$ and $u_{\xi}$ uniquely determined by
	\begin{equation}\label{eq:cara_1}
		\begin{cases}
		\dot{u}_{\xi}(s)=\breve{L}(\xi(s),\dot{\xi}(s),u_{\xi}(s)),\quad a.e. s\in[0,t],&\\
		u_{\xi}(0)=-u,&
		\end{cases}
	\end{equation}
	we define 
	\begin{align*}
		\eta(s)=\xi(t-s),\quad u_{\eta}(s)=-u_{\xi}(t-s),\quad s\in[0,t].
	\end{align*}
	Then $\xi\in\Gamma^t_{y,x}$ is a minimal curve for $h_{\breve{L}}(t,y,x,-u)$ with $u_{\xi}$ uniquely determined by \eqref{eq:cara_1} if and only if $\eta\in\Gamma^t_{x,y}$ is a minimal curve for $\breve{h}_L(t,x,y,u)$ with $u_\eta$ satisfying \eqref{eq:caratheodory_L_2} with respect to $\eta$.
\end{Lem}

\begin{Pro}\label{D_breave_h_L}
	Let $x,y\in \R^n$, $t>0$ and $u\in\R$. Let $\xi\in\Gamma^t_{x,y}$ be a minimizer for $\breve{h}_L(t,x,y,u)$ and let $u_{\xi}$ be determined by \eqref{eq:caratheodory_L_2} with $u_{\xi}(t)=u$. Then, we have the following relations:
	\begin{equation}\label{eq_relation_differential_breve_h_L}
		\begin{split}
			e^{-\int^t_0L_u(\xi,\dot{\xi},u_{\xi})ds}\cdot L_v(\xi(t),\dot{\xi}(t),u_{\xi}(t))\in D^+_y\breve{h}_L(t,x,y,u),\\
	-L_v(\xi(0),\dot{\xi}(0),u_{\xi}(0))\in D^+_x\breve{h}_L(t,x,y,u),\\
	1-e^{-\int^t_0L_u(\xi,\dot{\xi},u_{\xi})ds}\in D^+_u\breve{h}_L(t,x,y,u).
		\end{split}
	\end{equation}
	In particular, if the right sides of the relations above are singletons when $\breve{h}_L$ is differentiable, then the inclusions  become equalities respectively.
\end{Pro}

\subsection{Quantitative semiconcavity and convexity estimates of $h_L$}

In this section we provide quantitative estimates for the semiconcavity and semiconvexity of the fundamental solution.

Fix $\lambda>0$, $T>0$ and $M>0$. Let $x,y\in\mathbb{R}^{n}$, $|u|\leqslant M$, $0<t\leqslant T$ and $|x-y|\leqslant\lambda t$. For any minimizer $\xi\in\Gamma^{t}_{x,y}$ of \eqref{eq:fundamental_solution}, Theorem~2.5 and Proposition~A.4 of \cite{CCJWY2020} imply the existence of a continuous function $F:\mathbb{R}\times[0,+\infty)\times[0,+\infty)\to[0,+\infty)$, nondecreasing in both $t$ and $r$ and superlinear in $r$, such that
\begin{equation}\label{eq:Lip_xi_dot_1}
    \sup_{s\in[0,t]}\bigl\{|u_{\xi}(s)|,|\dot{\xi}(s)|\bigr\}\leqslant F(M,T,\lambda).
\end{equation}
For the applications considered in this paper we assume $\lambda,M>0$ and $0<T\leqslant1$. Hence we may suppose there exists a constant $C>0$ satisfying
\begin{equation}\label{eq:Lip_xi_dot_2}
    \sup_{s\in[0,t]}\bigl\{|u_{\xi}(s)|,|\dot{\xi}(s)|\bigr\}\leqslant C.
\end{equation}
Moreover, from \eqref{eq:Lip_xi_dot_2} we obtain
\[
    |\xi(s)-x|\leqslant\int^{t}_{0}|\dot{\xi}(s)|\,ds\leqslant C,\qquad s\in[0,t].
\]
Consequently, there exists a compact set $\mathbf{K}\subset\mathbb{R}^{n}\times\mathbb{R}^{n}\times\mathbb{R}$ such that
\begin{equation}\label{eq:K1}
	\bigl\{(\xi(s),\dot{\xi}(s),u_{\xi}(s)):s\in[0,t]\bigr\}\subset\mathbf{K}.
\end{equation}

\subsubsection{Semiconcavity of the fundamental solution} \label{sec:semiconcave}

For any $z,w\in\R^n$ and $h,k\in\R$, we need to give an upper bound of
\begin{align*}
	I=h_L(t+h,x+z,y+w,u+k)+h_L(t-h,x-z,y-w,u-k)-2h_L(t,x,y,u)
\end{align*}
in terms of $|z|^2+|w|^2+|h|^2+|k|^2$. Now, let $\xi\in\Gamma^t_{x,y}$ be a minimizer of \eqref{eq:fundamental_solution} with $u_{\xi}$ uniquely determined by \eqref{eq:Caratheodory_ODE}. We define
\begin{align*}
	\xi^{\pm}(s)=\xi\left(\frac{st}{t\pm h}\right)\pm\frac s{t\pm h}\cdot w\pm\frac{t\pm h-s}{t\pm h}\cdot z, \quad s\in[0,t\pm h].
\end{align*}
Then $\xi^{\pm}\in\Gamma^{t\pm h}_{x\pm z,y\pm w}$. We denote by $u^{\pm}(\cdot)=u_{\xi^{\pm}}(\cdot,u\pm k)$ the solutions of \eqref{eq:Caratheodory_ODE} with respect to $\xi^{\pm}$ and initial conditions $u^{\pm}(0)=u\pm k$. Set
	\begin{align*}
		\eta^{\pm}(\tau)=\xi^{\pm}\left(\frac{t\pm h}t\cdot \tau\right),\quad v^{\pm}(\tau)=u^{\pm}\left(\frac{t\pm h}t\cdot \tau\right),\quad \tau\in[0,t].
	\end{align*}
	Then, we have that
	\begin{equation}\label{eq:rescale1}
			\eta^{\pm}(\tau)=\xi(\tau)\pm\frac{\tau w}t\pm \frac{(t-\tau)z}t,\quad \dot{\eta}^{\pm}(\tau)=\dot{\xi}(\tau)\pm\frac{w-z}{t},\quad \tau\in[0,t]
	\end{equation}
	and $v^{\pm}$ satisfy the ordinary differential equations
	\begin{equation}\label{eq:rescale2}
		\dot{v}^{\pm}(\tau)=\frac{t\pm h}tL\left(\eta^{\pm}(\tau),\frac{t}{t\pm h}\dot{\eta}^{\pm}(\tau),v^{\pm}(\tau)\right),\quad \tau\in[0,t],
	\end{equation}
	with initial conditions $v^{\pm}(0)=u\pm k$ respectively. 
	
\begin{Lem}\label{rescale_u}
	There exists a constant $C>0$ depending only on $\mathbf{K}$ such that
	\begin{align*}
		|v^+(\tau)-v^-(\tau)|\leqslant C\big(|k|+|h|+|z|+|w|\big),\quad \forall\tau\in[0,t].
	\end{align*}
\end{Lem}

\begin{Lem}\label{rescale_u2}
	There exists $C>0$ such that the following holds
	\begin{align*}
		|&\,v^+(\tau)+v^-(\tau)-2u_{\xi}(\tau)|\leqslant  \frac {C}t(|h|^2+|z|^2+|w|^2)+C|k|^2,\quad\quad \forall\tau\in[0,t],\\
		\int^t_0&\, \left\{L\left(\eta^+,\frac t{t+h}\dot{\eta}^+,v^+\right)+L\left(\eta^-,\frac t{t-h}\dot{\eta}^-,v^-\right)-2L(\xi,\dot{\xi},u_{\xi})\right\}\ d\tau\\
		&\,\leqslant \frac {C}t(|h|^2+|z|^2+|w|^2)+C|k|^2.
	\end{align*}
\end{Lem}

Now we begin our semiconcavity estimate of the fundamental solution.

\begin{The}[Semiconcavity of fundamental solutions]\label{semiconcavity}
	For any $z$, $w\in\R^n$ and $h$, $k\in\R$ with $|z|,|w|\leqslant\lambda t$, $|k|\leqslant\delta$, $|h|<\frac t2$, there exist constants $C>0$ such that
	\begin{align*}
		h_L(&\,t+h,x+z,y+w,u+k)+h_L(t-h,x-z,y-w,u-k)-2h_L(t,x,y,u)\\
		\leqslant&\, \frac {C}t(|h|^2+|z|^2+|w|^2)+C|k|^2.
	\end{align*}
\end{The}

\subsubsection{Some regularity lemmata}

Fix $\lambda>0$. Let $x,y_1,y_2\in\R^n$, $u_i\in\R$ and $0<t\leqslant 1$ such that $|y_i-x|\leqslant\lambda t,\,|u_i|\leqslant \delta$ for $i=1,2$. Similar to the previous subsection, there exists a compact subset $\mathbf{K}_{\lambda,\delta}\subset\R^n\times\R^n\times\R$ such that, if $\xi_i$ is a minimal curve for $h_L(t,x,y_i,u_i)$ with $u_{\xi_i}$ determined by \eqref{eq:Caratheodory_ODE} with initial condition $u_{\xi_i}(0)=u_i$, $i=1,2$. then
\begin{align*}
	\mbox{\rm co}\,\{(\xi_i(s),\dot{\xi}_i(s),u_{\xi_i}(s)), s\in[0,t],i=1,2\}\subset\mathbf{K}_{\lambda,\delta}.
\end{align*}
Let $p_i=L_v(\xi_i,\dot{\xi}_i,u_{\xi_i})$ the dual arc of $\xi_i$, $i=1,2$.

\begin{Lem}[Main Regularity Lemma]\label{regularity_main}
	There exists $t_{\lambda,\delta}>0$ such that for all $t\in(0,t_{\lambda,\delta}]$, we have
	\begin{align}
		\|\xi_1-\xi_2\|^2\leqslant&\,\frac Ct|y_1-y_2|^2+C|u_1-u_2|^2,\label{hu:xi}\\
		\int^t_0|p_2-p_1|^2d\tau\leqslant&\,\frac Ct|y_1-y_2|^2+C|u_1-u_2|^2,\nonumber\\
		\|u_{\xi_1}-u_{\xi_2}\|^2\leqslant&\, C(|y_1-y_2|^2+|u_1-u_2|^2),\nonumber\\
		\int^t_0|\dot{\xi}_2-\dot{\xi}_1|^2d\tau\leqslant&\,\frac Ct|y_1-y_2|^2+C|u_1-u_2|^2.\label{hu:dotxi}
	\end{align}
	where the positive constants $t_{\lambda,\delta}$ and $C$ depend on $\mathbf{K}_{\lambda,\delta}$, and $\|\cdot\|$ stands for $C_0$-norm.
\end{Lem}

Similar to Lemma \ref{regularity_main}, we also need a regularity lemma on a small perturbation on $t$-variable.

Fix $\lambda>0$. Given $x,y,z\in\R^n$, $u,k,h\in\R$ and $0<t\leqslant 1$, let $t_{\pm}=t\pm h,\,y_{\pm}=y\pm z,\,u_{\pm}=u\pm k$, $\xi_{\pm}$ be a minimal curve for $h_L(t_{\pm},x,y_{\pm},u_{\pm})$ and $u_{\xi{\pm}}$ be determined by \eqref{eq:Caratheodory_ODE} with initial condition $u_{\xi_{\pm}}(0)=u\pm k$. We also denote by $p_{\pm}=L_v(\xi_{\pm},\dot{\xi}_{\pm},u_{\xi_{\pm}})$ the dual arc of $\xi_{\pm}$.  We also suppose that $|y_{\pm}-x|\leqslant\lambda t$, $|z|\leqslant\lambda t$, $|h|<\frac{t}{2}$ and $|u_{\pm}|\leqslant \delta$. Therefore, as explained previously, there exists a compact subset $\mathbf{K}_{\lambda,\delta}\subset\R^n\times\R^n\times\R$ such that
\begin{align*}
	\mbox{co}\,\{(\xi_{\pm}(s),\dot{\xi}_{\pm}(s),u_{\xi_{\pm}}(s)), s\in[0,t]\}\subset\mathbf{K}_{\lambda,\delta}.
\end{align*}

Denote 
$$\eta_\pm(\tau)=\xi_\pm\left(\frac{t\pm h}{t}\tau\right),\quad v_\pm(\tau)=u_\pm\left(\frac{t\pm h}{t}\tau\right),\quad \tau\in[0,t].$$
Then $v_\pm$ satisfy the ordinary differential equations
\begin{equation}\label{hu:eq}
\dot{v}_{\pm}(\tau)=\frac{t\pm h}tL\left(\eta_{\pm}(\tau),\frac{t}{t\pm h}\dot{\eta}_{\pm}(\tau),v_{\pm}(\tau)\right),\quad \tau\in[0,t],
\end{equation}
	with initial conditions $v_{\pm}(0)=u_\pm $ respectively.
	Define
\begin{align*}
	\xi(\tau)=\frac{\eta_+(\tau)+\eta_-(\tau)}2,\quad \tau\in[0,t],
\end{align*}
then $\xi\in\Gamma^t_{x,y}$.

\begin{Lem}\label{lemma}
There exists $t_{\lambda,\delta}>0$ such that for all $t\in(0,t_{\lambda,\delta}]$
\begin{align}
&\|\eta_+-\eta_-\|^2\leqslant\frac{C}{t}(|h|^2+|z|^2)+C|k|^2,\label{hu:eta}\\
&\int_0^t|\dot{\eta}_+-\dot{\eta}_-|^2d\tau\leqslant\frac{C}{t}(|h|^2+|z|^2)+C|k|^2,\label{hu:doteta}\\
&\|v_+-v_-\|^2\leqslant C(|h|^2+|z|^2+|k|^2),\label{hu:v}\\
&\|v_++v_--2u_\xi\|\leqslant\frac{C}{t}(|h|^2+|z|^2)+C|k|^2,\label{hu:vv}
\end{align}
where the constant $C>0$ depends on the compact subset $\mathbf{K}_{\lambda,\delta}$ and $u_{\xi}$ be determined by \eqref{eq:Caratheodory_ODE} with initial condition $u_{\xi}(0)=u$.
\end{Lem}

\subsubsection{Convexity estimate for short time}

\begin{The}\label{hu:uniform covex}
There exists $t_{\lambda,\delta}>0$ such that for any $x\in\mathbb{R}^n$, the function $(t,y,u)\mapsto h_L(t,x,y,u)$ is semiconvex on the cone
\begin{align*}
	S_{\lambda,\delta}=\{(t,y,u): t\in(0,t_{\lambda,\delta}], |y-x|<\lambda t,|u|<\delta\}.
\end{align*}
More precisely, there exists $C_1>0$ depending on the compact subset $\mathbf{K}_{\lambda,\delta}$ defined previously such that for $|h|<\frac t2$, $|z|<\lambda t$ and $|k|<\delta$,
\begin{equation}\label{convex}
	\begin{split}
		h_L(&\,t+h,x,y+z,u+k)+h_L(t-h,x,y-z,u-k)-2h_L(t,x,y,u)\\
		\geqslant&\,-\frac{C_1}{t}(|h|^2+|z|^2)-C_1|k|^2.
	\end{split}
\end{equation}

In particular, there exist $t^\prime_{\lambda,\delta}$ and $C_2>0$ also depending on the compact subset $\mathbf{K}_{\lambda,\delta}$, such that for all $t\in (0,t^\prime_{\lambda,\delta}]$ the function $h_L(t,x,\cdot,u)$ is uniformly convex on $B(x,\lambda t)$, i.e., for all $|y-x|<\lambda t$ and $|z|<\lambda t$ we have that
\begin{equation}\label{uniform convex}
	h_L(t,x,y+z,u)+h_L(t,x,y-z,u)-2h_L(t,x,y,u)
		\geqslant\,\frac{C_2}{t}|z|^2.
\end{equation}
\end{The}

\begin{Cor}\label{cor:convexity_t}
Fix $\lambda>0$. Let $x,y\in\R^n$, $u,h\in\R$ and $0<t\leqslant 1$ such that $|y-x|\leqslant\lambda t$, $|u|\leqslant \delta$ and $|h|<\frac{t}{2}$. Suppose that $\xi_\pm$ is a minimal curve for $h_L(t_\pm,x,y,u)$ and
\begin{align*}
	\inf_{s\in [0,t+h]}|\dot{\xi}_+|^2+\inf_{s\in [0,t-h]}|\dot{\xi}_-|^2\geqslant C_1,
\end{align*}
where $C_1>0$ is a constant depending on the compact subset $\mathbf{K}_{\lambda,\delta}$ defined previously. Then for any $x\in\mathbb{R}^n$, the function $t\rightarrow h_L(t,x,y,u)$ is uniformly convex. Thus, there exists $C_2$ such that
\begin{align*}
	h_L(t+h,x,y,u)+h_L(t-h,x,y,u)-2h_L(t,x,y,u)\geqslant\,\frac{C_2}{t}|h|^2.
\end{align*}
\end{Cor}

It is not hard to show that the combination of the previous results leads to the following theorem.

\begin{The}\label{regularity}
	Fix $\lambda>0$, $x\in\R^n$, then there exist constant $t_{\lambda}$ and $C_{\lambda}>0$ depending on $\lambda,\,x$ and $u$ such that
	\begin{enumerate}[\rm (a)]
	\item The map $y\mapsto h_L(t,x,y,u)$ is both locally semiconcave and uniformly convex on $B(x,\lambda t)$ for $0<t<t_{\lambda}$, and the constant of convexity is $C_{\lambda}/t$. A similar result also holds for the map $x\mapsto h_L(t,x,y,u)$.
	\item The map $(t,y,u)\mapsto h_L(t,x,y,u)$ is locally $C^{1,1}$ on the set
	$$S=\{(t,y,u)\in \mathbb{R}\times \mathbb{R}^n\times\mathbb{R}:0<t<t_\lambda,\,|y-x|<\lambda t,\,|u|<\delta\}.$$
	Moreover, for all $(t,y,u)\in S$,
	\begin{align*}
	                 D_th_L(t,x,y,u)=&\,-e^{\int^t_0L_u(\xi,u_{\xi},\dot{\xi})\ ds}E(t)=-H(\xi(t),u_{\xi}(t),p(t)),\\
		D_yh_L(t,x,y,u)=&\,L_v(\xi(t),u_{\xi}(t),\dot{\xi}(t)),\\
		D_xh_L(t,x,y,u)=&\,-e^{\int^t_0L_u(\xi,u_{\xi},\dot{\xi})d\tau}\cdot L_v(\xi(0),u_{\xi}(0),\dot{\xi}(0)),\\
		D_uh_L(t,x,y,u)=&\,e^{\int^t_0L_u(\xi,u_{\xi},\dot{\xi})\ d\tau}-1,
	\end{align*}
	where $\xi\in\Gamma^t_{x,y}$ is the unique minimizer for $h_L(t,x,y,u)$ with $u_{\xi}$ determined by \eqref{eq:Caratheodory_ODE}.
	\end{enumerate}
\end{The}

\section{Proofs of the statements in Section \ref{sec:statement_regularity}}

\subsection{First order regularity}

\begin{proof}[Proof of Proposition \ref{D_h_L}]
The essential idea for the first order estimates is to analyze the associated variational equation. Suppose $x,y\in\R^n$, $t>0$ and $u\in\R$. Let $\xi\in\Gamma^t_{x,y}$ be a minimizer of \eqref{eq:fundamental_solution} and let be $u_{\xi}(s)$ the unique solution of \eqref{eq:Caratheodory_ODE}. 

\medskip
	
\noindent (a) For any $\lambda>0$ and $0\not=\theta\in\R^n$, we define
\begin{align*}
	\xi_{\lambda}(s)=\xi(s)+\frac st\lambda\theta,\quad s\in[0,t],
\end{align*}
and let $u_{\xi_{\lambda}}$ be the unique solution of \eqref{eq:Caratheodory_ODE} with respect to  $\xi_{\lambda}$. If we denote 
\begin{gather*}
	f(s)=\frac{\partial}{\partial\lambda}u_{\xi_{\lambda}}(s)\bigg\vert_{\lambda=0},\quad h(s)=L_u(\xi(s),\dot{\xi}(s),u_{\xi}(s)),\\
	g(s)=\left\langle L_x(\xi(s),\dot{\xi}(s),u_{\xi}(s)),\frac st\theta\right\rangle+\left\langle L_v(\xi(s),\dot{\xi}(s),u_{\xi}(s)),\frac 1t\theta\right\rangle,
\end{gather*}
then we have that $f(0)=0$ since $u_{\xi_{\lambda}}(0)$ is constant with respect to $\lambda$. It follows that
\begin{equation}\label{eq:variational_eqn}
	\dot{f}(s)=g(s)+h(s)f(s),\quad  s\in[0,t].
\end{equation}
By solving the variational equations above, we have
\begin{equation}\label{eq:sol_variational_eqn}
	f(s)=e^{\int^s_0h(r)\ dr}\cdot\int^s_0e^{-\int^r_0h(\tau)\ d\tau}\cdot g(r)\ dr,\quad  s\in[0,t].
\end{equation}
To prove \eqref{eq:p_in_superdiff}, it suffices to show that
\begin{align*}
	\limsup_{\lambda\to0^+}\frac{h_L(t,x,y+\lambda\theta,u)-h_L(t,x,y,u)-\langle p_y(t),\lambda\theta\rangle}{\lambda}\leqslant0,
\end{align*}
where $p_y(t)$ is determined by \eqref{eq:p_in_superdiff}. Indeed, we have that 
\begin{align}
	&\,\limsup_{\lambda\to0^+}\frac{h_L(t,x,y+\lambda\theta,u)-h_L(t,x,y,u)}{\lambda}\notag\\
		\leqslant&\,\lim_{\lambda\to0^+}\frac 1{\lambda}\int^t_0\{L(\xi_{\lambda},\dot{\xi}_{\lambda},u_{\xi_{\lambda}})-L(\xi,\dot{\xi},u_{\xi})\}\ ds=\int^t_0g(s)+h(s)f(s)\ ds\notag\\
		=&\,e^{\int^t_0h(\tau)d\tau}\cdot\int^t_0e^{-\int^s_0h(\tau)d\tau}g(s)\ ds\notag\\
		=&\,e^{\int^t_0L_u(\xi(\tau),\dot{\xi}(\tau),u_{\xi}(\tau))d\tau}\cdot\int^t_0\frac d{ds}\left\langle e^{-\int^s_0L_u(\xi(\tau),\dot{\xi}(\tau),u_{\xi}(\tau))d\tau}L_v(\xi(s),\dot{\xi}(s),u_{\xi}(s)),\frac st\theta\right\rangle\ ds\label{eq:add1}\\
		=&\,\langle L_v(\xi(t),\dot{\xi}(t),u_{\xi}(t)),\theta\rangle\notag.
\end{align}
This leads to \eqref{eq:p_in_superdiff} and the last assertion of (a) follows directly.

\medskip

\noindent (b) For any $\lambda>0$ and $0\not=\theta\in\R^n$, we define
\begin{align*}
	\xi_{\lambda}(s)=\xi(s)+\frac {t-s}t\lambda\theta,\quad s\in[0,t],
\end{align*}
and let $u_{\xi_{\lambda}}$ be the unique solution of \eqref{eq:Caratheodory_ODE} with respect to $\xi_{\lambda}$. Similar to the proof of (a), by setting
\begin{gather*}
	f(s)=\frac{\partial}{\partial\lambda}u_{\xi_{\lambda}}(s)\bigg\vert_{\lambda=0},\quad h(s)=L_u(\xi(s),\dot{\xi}(s),u_{\xi}(s)),\\
	g(s)=\left\langle L_x(\xi(s),\dot{\xi}(s),u_{\xi}(s)),\frac{t-s}t\theta\right\rangle-\left\langle L_v(\xi(s),\dot{\xi}(s),u_{\xi}(s)),\frac 1t\theta\right\rangle,
\end{gather*}
we conclude that $f$ satisfies the variational equation \eqref{eq:variational_eqn} with $g,h$ defined above, with $f(0)=0$, and it admits a unique solution in the form of \eqref{eq:sol_variational_eqn}. Invoking Herglotz equation \eqref{eq:Herglotz} and \eqref{eq:add1}, we obtain
\begin{equation}\label{eq:E1}
	\begin{split}
		&\,\frac d{ds}\left\langle e^{-\int^s_0L_u(\xi,\dot{\xi},u_{\xi})d\tau}L_v(\xi(s),\dot{\xi}(s),u_{\xi}(s)),\frac st\theta\right\rangle\\
	=&\,e^{-\int^s_0L_u(\xi,\dot{\xi},u_{\xi})d\tau}\left\{\left\langle L_v(\xi(s),\dot{\xi}(s),u_{\xi}(s)),\frac 1t\theta\right\rangle+\left\langle \frac d{ds}L_v-L_uL_v,\frac st\theta\right\rangle\right\}
	\end{split}
\end{equation}
and
\begin{equation}\label{eq:E2}
	\begin{split}
		&\,e^{\int^t_0L_u(\xi,\dot{\xi},u_{\xi})d\tau}\int^t_0e^{-\int^s_0L_u(\xi,\dot{\xi},u_{\xi})d\tau}\langle L_x,\theta\rangle\ ds\\
	=&\,e^{\int^t_0L_u(\xi,\dot{\xi},u_{\xi})d\tau}\int^t_0e^{-\int^s_0L_u(\xi,\dot{\xi},u_{\xi})d\tau}\left\langle \frac d{ds}L_v-L_uL_v,\theta\right\rangle\ ds\\
	=&\,e^{\int^t_0L_u(\xi,\dot{\xi},u_{\xi})d\tau}\cdot\left\{e^{-\int^s_0L_u(\xi,\dot{\xi},u_{\xi})d\tau}\langle L_v(\xi(s),\dot{\xi}(s),u_{\xi}(s)),\theta\rangle\right\}\bigg\vert_0^t\\
	=&\,\langle L_v(\xi(t),\dot{\xi}(t),u_{\xi}(t)),\theta\rangle-e^{\int^t_0L_u(\xi,\dot{\xi},u_{\xi})d\tau}\cdot\langle L_v(\xi(0),\dot{\xi}(0),u_{\xi}(0)),\theta\rangle.
	\end{split}
\end{equation}
By combining \eqref{eq:E1} and \eqref{eq:E2}, we conclude
\begin{align*}
	&\,\limsup_{\lambda\to0^+}\frac{h_L(t,x+\lambda\theta,y,u)-h_L(t,x,y,u)}{\lambda}\\
		\leqslant&\,\lim_{\lambda\to0^+}\frac 1{\lambda}\int^t_0\{L(\xi_{\lambda},\dot{\xi}_{\lambda},u_{\xi_{\lambda}})-L(\xi,\dot{\xi},u_{\xi})\}\ ds=\int^t_0g(s)+h(s)f(s)\ ds\\
		=&\,e^{\int^t_0L_u(\xi,\dot{\xi},u_{\xi})d\tau}\cdot\int^t_0e^{-\int^s_0L_u(\xi,\dot{\xi},u_{\xi})d\tau}g(s)\ ds\\
		=&\,-e^{\int^t_0L_u(\xi,\dot{\xi},u_{\xi})d\tau}\cdot\langle L_v(\xi(0),\dot{\xi}(0),u_{\xi}(0)),\theta\rangle.
\end{align*}
This completes the proof of \eqref{eq:x_in_superdiff} and the last assertion of (b) follows.

\medskip

\noindent (c) For any $k\in\R$, let $u_{\xi,k}$ be the unique solution of \eqref{eq:Caratheodory_ODE} with respect to $\xi$ with initial condition $u_{\xi,k}(0)=u+k$. Set
\begin{align*}
	f(s)=\frac{\partial}{\partial k}u_{\xi,k}(s)\bigg\vert_{k=0},\quad h(s)=L_u(\xi(s),\dot{\xi}(s),u_{\xi}(s)),
\end{align*}
then we have that $f(0)=1$. Thus, by solving variational equation
\begin{align*}
	\dot{f}(s)=h(s)f(s),\quad s\in[0,t],
\end{align*}
we have that
\begin{align*}
	f(s)=e^{\int^s_0L_u(\xi(\tau),\dot{\xi}(\tau),u_{\xi}(\tau))\ d\tau},\quad s\in[0,t].
\end{align*}
It follows
\begin{align*}
	&\,\limsup_{k\to0}\frac{h(t,x,y,u+k)-(t,x,y,u)}{k}\\
	\leqslant&\,\lim_{k\to0}\frac 1{k}\int^t_0\{L(\xi,\dot{\xi},u_{\xi,k})-L(\xi,\dot{\xi},u_{\xi})\}\ ds=\int^t_0h(s)f(s)\ ds\\
	=&\,e^{\int^t_0L_u(\xi,\dot{\xi},u_{\xi})\ d\tau}-1,
\end{align*}
and this leads to \eqref{eq:u_in_superdiff}.
\end{proof}

\begin{proof}[Proof of Lemma \ref{energy}]
We work on $\R^n$. Let $(\xi,u_{\xi})$ be an extremal for \eqref{eq:Herglotz} on $[0,t]$, by direct calculation, we have that
	\begin{align*}
		&\,\frac d{ds}\left\{e^{-\int^s_0L_u(\xi,\dot{\xi},u_{\xi})\ d\tau}\cdot(L(\xi(s),\dot{\xi}(s),u_{\xi}(s))- L_v(\xi(s),\dot{\xi}(s),u_{\xi}(s))\cdot\dot{\xi}(s)\right\}\\
		=&\,e^{-\int^s_0L_u(\xi,\dot{\xi},u_{\xi})\ d\tau}\cdot[L_x\cdot\dot{\xi}(s)+L_u\cdot\dot{u}_{\xi}(s)+L_v\cdot\ddot{\xi}(s)-\frac d{ds}L_v\cdot\dot{\xi}(s)-L_v\cdot\ddot{\xi}(s)]\\
		&\,-L_u\cdot e^{-\int^s_0L_u(\xi,\dot{\xi},u_{\xi})\ d\tau}\cdot[L-L_v\cdot\dot{\xi}(s)]=0.
	\end{align*}
	It follows that $E$ is constant on $[0,t]$. The case for $E^*$ is similar.
\end{proof}

\begin{proof}[Proof of Proposition \ref{D_t}]
For any $h\in\R$ such that $t+h>0$, we define
	\begin{align*}
		\xi_h(\tau)=\xi\left(\frac t{t+h}\tau\right),\quad \tau\in[0,t+h],
	\end{align*}
	and $u_{\xi_h}:[0,t+h]\to\R$ is defined by 
	\begin{align*}
		\dot{u}_{\xi_h}(\tau)=L(\xi_h(\tau),\dot{\xi}_h(\tau),u_{\xi_h}(\tau)),\quad \tau\in[0,t+h]
	\end{align*}
	with initial condition $u_{\xi_h}(0)=u$. Let $v_{\xi_h}(s)=u_{\xi_h}(\frac{t+h}ts)$. then we have that
	\begin{align*}
		\int^{t+h}_0L(\xi_h(\tau),\dot{\xi}_h(\tau),u_{\xi_h}(\tau))\ d\tau=\frac{t+h}t\int^t_0L\left(\xi(s),\frac{t}{t+h}\dot{\xi}(s),v_{\xi_h}(s)\right)\ ds,
	\end{align*}
	and $v_{\xi_h}$ satisfies
	\begin{align*}
		\dot{v}_{\xi_h}(s)=\frac{t+h}tL\left(\xi(s),\frac{t}{t+h}\dot{\xi}(s),v_{\xi_h}(s)\right),\quad \tau\in[0,t],
	\end{align*}
	with $v_{\xi_h}(0)=u$. Then, the variational equation becomes
	\begin{align*}
		\dot{f}(s)=g(s)+h(s)f(s),\quad s\in[0,t],
	\end{align*}
	where 
	\begin{align*}
	    f(s);=&\,\frac d{dh}v_{\xi_h}(s)|_{h=0},\\
		g(s):=&\,\frac 1t(L(\xi(s),\dot{\xi}(s),u_{\xi}(s))-\langle L_v(\xi(s),\dot{\xi}(s),u_{\xi}(s)),\dot{\xi}(s)\rangle,\\
		h(s):=&\,L_u(\xi(s),\dot{\xi}(s),u_{\xi}(s)).
	\end{align*}
	Solving the linear ordinary equation above with $f(0)=0$, we have that
	\begin{align*}
		f(s)=e^{\int^s_0h(r)\ dr}\cdot\int^s_0e^{-\int^r_0h(\tau)\ d\tau}\cdot g(r)\ dr.
	\end{align*}
	Therefore, in view of Lemma \ref{energy}, we have that
	\begin{align*}
		&\/\limsup_{h\to0}\frac{h_L(t+h,x,y,u)-h_L(t,x,y,u)}{h}\\
		\leqslant&\,\lim_{h\to0}\frac 1h\int^t_0\left\{L\left(\xi(s),\frac{t}{t+h}\dot{\xi}(s),v_{\xi_h}(s)\right)-L(\xi(s),\dot{\xi}(s),u_{\xi}(s))\right\} ds\\
		&\,+\lim_{h\to0}\frac 1t\int^t_0L\left(\xi(s),\frac{t}{t+h}\dot{\xi}(s),v_{\xi_h}(s)\right)\ ds\\
		=&\,\int^t_0L_u(\xi,\dot{\xi},u_{\xi})f-\frac 1tL_v(\xi,\dot{\xi},u_{\xi})\cdot\dot{\xi}+\frac 1tL(\xi,\dot{\xi},u_{\xi})\ ds=\int^t_0hf+g\ ds\\
	=&\, \int^t_0\dot{f}\ ds=f(t)-f(0)=e^{\int^t_0h(r)\ dr}\cdot\left\{\int^t_0e^{-\int^s_0h(r)\ dr}\cdot g(s)\ ds\right\}\\
	=&\,-e^{\int^t_0L_u(\xi,\dot{\xi},u_{\xi})\ ds}\cdot\frac 1t\int^t_0 E(s)\ ds=-e^{\int^t_0L_u(\xi,\dot{\xi},u_{\xi})\ ds}E(t).
	\end{align*}
	Invoking Lemma \ref{energy}, we have that $-e^{\int^t_0L_u(\xi,\dot{\xi},u_{\xi})\ ds}E(t)\in D^+_th_L(t,x,y,u)$. 
	
	If $h_L(t,x,y,u)$ is differentiable at $t$, then
	\begin{align*}
		D_th_L(t,x,y,u)=-e^{\int^t_0L_u(\xi,\dot{\xi},u_{\xi})\ ds}E(t)=-H(\xi(t),p(t),u_{\xi}(t)),
	\end{align*}
	which completes the proof of \eqref{diff_of_fund_sol_t}.
\end{proof}

\begin{proof}[Proof of Lemma \ref{eq:equiv_positive}]
	Fix $x,y\in\R^n$, $t>0$. Let $\xi\in\Gamma^t_{y,x}$ and $u_{\xi}$ is determined by \eqref{eq:cara_1}, i.e.,
	\begin{align*}
		h_{\breve{L}}(t,y,x,-u)=\int^t_0\breve{L}(\xi(s),\dot{\xi}(s),u_{\xi}(s))\ ds.
	\end{align*}
	Notice that $\dot{\eta}(s)=-\dot{\xi}(t-s)$, $\dot{u}_{\eta}(s)=\dot{u}_{\xi}(t-s)$ and $u_{\eta}$ satisfies \eqref{eq:caratheodory_L_2} with respect to $\eta$. Therefore
	\begin{align*}
		\breve{h}_L(t,x,y,u)\leqslant&\,\int^t_0L(\eta(s),\dot{\eta}(s),u_{\eta}(s))\ ds\\
		=&\,\int^t_0L(\eta(t-\tau),\dot{\eta}(t-\tau),u_{\eta}(t-\tau))\ d\tau\\
		=&\,\int^t_0\breve{L}(\xi(\tau),\dot{\xi}(\tau),u_{\xi}(\tau))\ d\tau=h_{\breve{L}}(t,y,x,-u).
	\end{align*}
	The opposite inequality can be obtained similarly.
\end{proof}

\begin{proof}[Proof of Proposition \ref{D_breave_h_L}]
Fix $x,y\in M$, $t>0$ and $u\in\R$. In view of Lemma \ref{eq:equiv_positive}, we have that $\breve{h}_L(t,x,y,u)=h_{\breve{L}}(t,y,x,-u)$ for all $x,y\in\R^n$, $t>0$ and $u\in\R$. Moreover, If $\xi\in\Gamma^t_{x,y}$ is a minimal curve for $\breve{h}_L(t,x,y,u)$ and $u_{\xi}$ is determined by \eqref{eq:caratheodory_L_2} with $u_{\xi}(t)=u$, then the curve $\eta\in\Gamma^t_{y,x}$ and $u_{\eta}:[0,t]\to\R$ defined by $\eta(s):=\xi(t-s)$, $s\in[0,t]$, is a minimal curve for $h_{\breve{L}}(t,y,x,-u)$, and $u_{\eta}$ defined by $u_{\eta}(s):=-u_{\xi}(t-s)$, $s\in[0,t]$ satisfies \eqref{eq:Caratheodory_ODE} associated to $\breve{L}$ with $u_{\eta}(0)=-u$. Notice that we have that
\begin{align*}
	e^{\int^t_0\breve{L}_u(\eta,\dot{\eta},u_{\eta})ds}=&\,e^{-\int^t_0L_u(\xi,\dot{\xi},u_{\xi})ds}\\
	\breve{L}_v(\eta(0),\dot{\eta}(0),u_{\eta}(0))=&\,-L_v(\xi(t),\dot{\xi}(t),u_{\xi}(t))\\
	\breve{L}_v(\eta(t),\dot{\eta}(t),u_{\eta}(t))=&\,-L_v(\xi(0),\dot{\xi}(0),u_{\xi}(0)),
\end{align*}
and the relations
\begin{align*}
	D^+_y\breve{h}_L(t,x,y,u)=&\,D^+_xh_{\breve{L}}(t,y,x,-u),\\
	D^+_x\breve{h}_L(t,x,y,u)=&\,D^+_yh_{\breve{L}}(t,y,x,-u),\\
	D^+_u\breve{h}_L(t,x,y,u)=&\,-D^+_uh_{\breve{L}}(t,y,x,-u).
\end{align*}
Then \eqref{eq_relation_differential_breve_h_L} follows and the last assertion is immediate.
\end{proof}

\subsection{Semiconcavity and convexity estimates}

\begin{proof}[Proof of Lemma \ref{rescale_u}]
By the definitions of $\eta^{\pm}$ and $v^{\pm}$ together with \eqref{eq:rescale1} and \eqref{eq:rescale2}, we have that
	\begin{align*}
		v^+(\tau)-&v^-(\tau)\\
		=&\,(v^+(0)-v^-(0))+\int^{\tau}_0L\left(\eta^+,\frac{t}{t+h}\dot{\eta}^+,v^+\right)-L\left(\eta^-,\frac{t}{t-h}\dot{\eta}^-,v^-\right)ds\\
		&\,+\frac{h}t\int^{\tau}_0L\left(\eta^+,\frac{t}{t+h}\dot{\eta}^+,v^+\right)+L\left(\eta^-,\frac{t}{t-h}\dot{\eta}^-,v^-\right)ds\\
		\leqslant&\,2k+I_1+I_2,
	\end{align*}
	where 
	\begin{align*}
	I_1&=\int^{\tau}_0L\left(\eta^+,\frac{t}{t+h}\dot{\eta}^+,v^+\right)-L\left(\eta^-,\frac{t}{t-h}\dot{\eta}^-,v^-\right)ds,\\
	I_2&=\frac{h}t\int^{\tau}_0L\left(\eta^+,\frac{t}{t+h}\dot{\eta}^+,v^+\right)+L\left(\eta^-,\frac{t}{t-h}\dot{\eta}^-,v^-\right)ds.
	\end{align*}
	
	By \eqref{eq:K1} we get
	\begin{align*}
	|\eta^\pm(\tau)|\leqslant C_1,\quad\left|\frac{t}{t\pm h}\dot{\eta}^\pm(\tau)\right|\leqslant C_2,\quad |v^\pm(\tau)|=\big|u^{\pm}(\frac{t\pm h}t\cdot \tau)\big|\leqslant C_3.
	\end{align*}
Note that
\begin{align*}
	L\left(\eta^\pm,\frac{t}{t{\pm} h}\dot{\eta}^\pm,0\right)-K|v^{\pm}|\leqslant &L\left(\eta^\pm,\frac{t}{t\pm h}\dot{\eta}^\pm,v^\pm\right)\\
	\leqslant &L\left(\eta^\pm,\frac{t}{t{\pm} h}\dot{\eta}^\pm,0\right)+K|v^{\pm}|.
\end{align*}
Thus, we have
\[
\left|L\left(\eta^\pm(s),\frac{t}{t\pm h}\dot{\eta}^\pm(s),v^\pm(s)\right)\right|\leqslant C_4+K\cdot C_3,\quad \forall s\in[0,t],
\]
which implies
\[
I_2\leqslant C_5\cdot|h|.
\]

Since
\begin{align*}
	|\eta^+(\tau)-\eta^-(\tau)|=\left|\frac {2\tau}t(w-z)+2z\right|\leqslant4(|w|+|z|),
\end{align*}
and
\begin{equation}\label{eq:eta+-}
	\begin{split}
		&\,\left|\frac t{t+h}\dot{\eta}^+(\tau)-\frac t{t-h}\dot{\eta}^-(\tau)\right|\\
		=&\,\left|\left(\frac t{t+h}-\frac t{t-h}\right)\dot{\xi}(\tau)+\left(\frac 1{t+h}+\frac 1{t-h}\right)(w-z)\right|\\
		\leqslant&\,\frac {C_6}t(|h|+|w|+|z|),
	\end{split}
\end{equation}
we obtain that
\begin{equation}\label{eq:v+-}
	\begin{split}
			I_1=&\,\int^{\tau}_0L\left(\eta^+,\frac{t}{t+h}\dot{\eta}^+,v^+\right)-L\left(\eta^-,\frac{t}{t-h}\dot{\eta}^-,v^-\right)ds\\
		\leqslant& \,tC_7\cdot\sup_{\tau\in[0,t]}|\eta^+(\tau)-\eta^-(\tau)|+K\int^{\tau}_0|v^+-v^-|\ ds\\
                                  &\,+ tC_8\cdot\sup_{\tau\in[0,t]}\left|\frac t{t+h}\dot{\eta}^+(\tau)-\frac t{t-h}\dot{\eta}^-(\tau)\right|\\
		\leqslant&\,tC_9(|z|+|w|)+K\int^{\tau}_0|v^+-v^-|\ ds+C_{10}(|h|+|w|+|z|)\\
		\leqslant&\,C_{11}(|h|+|w|+|z|)+K\int^{\tau}_0|v^+-v^-|\ ds.
		\end{split}
\end{equation}

It follows that
\begin{align*}
	|v^+(\tau)-v^-(\tau)|\leqslant C_{12}(|k|+|h|+|z|+|w|)+K\int^{\tau}_0|v^+-v^-|\ ds.
\end{align*}
Invoking Gronwall inequality, we obtained
\begin{align*}
	|v^+(\tau)-v^-(\tau)|\leqslant C_{12}(|k|+|h|+|z|+|w|)\cdot\exp(Kt).
\end{align*}
Moreover, by \eqref{eq:v+-}, we have
\begin{equation}\label{eq:I_1}
	I_1\leqslant C_{11}\cdot(|h|+|w|+|z|)+tKC_{12}(|k|+|h|+|z|+|w|)\cdot\exp(Kt).
\end{equation}
Thus, we complete the proof.
\end{proof}

\begin{proof}[Proof of Lemma \ref{rescale_u2}]
	By the definitions of $\eta^{\pm}$ and $v^{\pm}$ together with \eqref{eq:rescale1} and \eqref{eq:rescale2}, we obtain
	\begin{align*}
		&\,v^+(\tau)+v^-(\tau)-2u_{\xi}(\tau)\\
		=&\,\int^{\tau}_0\left\{\frac{t+h}tL\left(\eta^+,\frac t{t+h}\dot{\eta}^+,v^+\right)+\frac{t-h}tL\left(\eta^-,\frac t{t-h}\dot{\eta}^-,v^-\right)-2L(\xi,\dot{\xi},u_{\xi})\right\} ds\\
		=&\, I_3+I_4,
		\end{align*}
	where 
\begin{align*}
	I_3=&\,\int^{\tau}_0\left\{L\left(\eta^+,\frac t{t+h}\dot{\eta}^+,v^+\right)+L\left(\eta^-,\frac t{t-h}\dot{\eta}^-,v^-\right)-2L(\xi,\dot{\xi},u_{\xi})\right\} ds,\\
	I_4=&\,\frac ht\int^{\tau}_0\left\{L\left(\eta^+,\frac t{t+h}\dot{\eta}^+,v^+\right)-L\left(\eta^-,\frac t{t-h}\dot{\eta}^-,v^-\right)\right\}ds.
\end{align*}
Due to \eqref{eq:I_1}, we have that
\begin{equation}\label{eq:I_4}
	\begin{split}
		I_4\leqslant&\,\frac {|h|}t\left\{C_1\cdot(|h|+|w|+|z|)+tKC_2(|k|+|h|+|z|+|w|)\cdot\exp(Kt)\right\}\\
		\leqslant&\,\frac {C_3}t(|h|^2+|z|^2+|w|^2)+C_4|k|^2.
	\end{split}
\end{equation}
	
It is clear that if a function $u:\R^n\to\R$ is Lipschitz with constant $K_1$ and semiconcave with constant $K_2$, then we have that
	\begin{align}\label{semiconcave}
		u(x)+u(y)-2u(z)\leqslant\frac {K_2}2|x-y|^2+K_1|x+y-2z|.
	\end{align}
	Since $L$ is of class $C^2$, then we have that
	\begin{equation}\label{eq:semiconcavity1}
		\begin{split}
			I_3\leqslant&\,\int^{\tau}_0\left|L\left(\eta^+,\frac t{t+h}\dot{\eta}^+,v^+\right)+L\left(\eta^-,\frac t{t-h}\dot{\eta}^-,v^-\right)-2L(\xi,\dot{\xi},u_{\xi})\right| ds\\
	\leqslant&\,C_5\cdot\int^t_0\left(|\eta^+-\eta^-|^2+|v^--v^+|^2+\left|\frac t{t+h}\dot{\eta}^+-\frac t{t-h}\dot{\eta}^-\right|^2\right)\ ds\\
		&\,+C_6\cdot\int^t_0(|\eta^++\eta^--2\xi|+|v^-+v^+-2u_{\xi}|)\ ds\\
		&\,+C_7\cdot\int^t_0\left|\frac t{t+h}\dot{\eta}^++\frac t{t-h}\dot{\eta}^--2\dot{\xi}\right|\ ds.
		\end{split}
	\end{equation}
	Note that
	\begin{align*}
		|\eta^+(\tau)+\eta^-(\tau)-2\xi(\tau)|=&\,0,\\
		\left|\frac t{t+h}\dot{\eta}^+(\tau)+\frac t{t-h}\dot{\eta}^-(\tau)-2\dot{\xi}(\tau)\right|=&\,\left|\frac{2h^2}{t^2-h^2}\dot{\xi}-\frac{2h}{t^2-h^2}(w-z)\right|\\
		\leqslant&\frac{C_8}{t^2}(|h|^2+|z|^2+|w|^2).
	\end{align*}
	By (\ref{eq:eta+-}), (\ref{eq:semiconcavity1}) and Lemma \ref{rescale_u}, it follows that
	\begin{align*}
		I_3\leqslant&\, C_9\cdot\int^t_0[4(|w|+|z|)^2+C_{10}^2(|k|+|z|+|w|+|h|)^2+\frac{C_{11}^2}{t^2}(|h|+|w|+|z|)^2]\ ds\\
		&\,+C_{12}\cdot\int^t_0 |v^-+v^+-2u_{\xi}|\ ds+C_{13}\cdot\int^t_0\frac{C^2_{14}}{t^2}\cdot(|h|^2+|z|^2+|w|^2)\ ds\\
		\leqslant&\,\frac {C_{15}}t(|h|^2+|z|^2+|w|^2)+C_{16}|k|^2+C_{17}\cdot\int_0^t\left|v^++v^--2u_\xi\right|\ ds.
	\end{align*}

Thus, we conclude that
\begin{align*}
	&\,|v^+(\tau)+v^-(\tau)-2u_{\xi}(\tau)|\leqslant I_3+I_4\\
	\leqslant&\,\frac {C_{18}}t(|h|^2+|z|^2+|w|^2)+C_{19}|k|^2+C_{17}\cdot\int_0^t\left|v^++v^--2u_\xi\right|\ ds.
\end{align*}
Invoking Gronwall inequality, we obtain
\begin{align*}
	|v^+(\tau)+v^-(\tau)-2u_{\xi}(\tau)|\leqslant \frac {C_{20}}t(|h|^2+|z|^2+|w|^2)+C_{21}|k|^2.
\end{align*}
The last estimate is a direct consequence and this completes the proof.
\end{proof}

\begin{proof}[Proof of Theorem \ref{semiconcavity}]
	The proof is based on the estimate of
\begin{align*}
	I\leqslant&\,\int^{t+h}_0L(\xi^+,\dot{\xi}^+,u^+)\ ds+\int^{t-h}_0L(\xi^-,\dot{\xi}^-,u^-)\ ds-2\int^{t}_0L(\xi,\dot{\xi},u_{\xi})\ ds\\
	=&\,\frac{t+h}t\int^t_0L\left(\eta^+,\frac{t}{t+h}\dot{\eta}^+,v^+\right)d\tau+\frac{t-h}t\int^t_0L\left(\eta^-,\frac{t}{t-h}\dot{\eta}^-,v^-\right)d\tau\\
	&\,-2\int^{t}_0L(\xi,\dot{\xi},u_{\xi})\ d\tau.
\end{align*}
Set
\begin{align*}
	I_1=&\,\frac ht\int^t_0\left\{L\left(\eta^+,\frac{t}{t+h}\dot{\eta}^+,v^+\right)-L\left(\eta^-,\frac{t}{t-h}\dot{\eta}^-,v^-\right)\right\}d\tau,\\
	I_2=&\,\int^t_0\left\{L\left(\eta^+,\frac{t}{t+h}\dot{\eta}^+,v^+\right)+L\left(\eta^-,\frac{t}{t-h}\dot{\eta}^-,v^-\right)-2L(\xi,\dot{\xi},u_{\xi})\right\}d\tau.
\end{align*}
Then we have that $I\leqslant I_1+I_2$. By  \eqref{eq:I_4}, we obtain
\begin{align*}
	I_1\leqslant C_1\cdot\left(\frac 1t+1\right)(|h|^2+|z|^2+|w|^2)+C_2|k|^2.
\end{align*}
The estimate of $I_2$ is obtained by Lemma \ref{rescale_u2}. The combination of the estimate of $I_1$ and $I_2$ leads to our conclusion.
\end{proof}

Due to Fenchel-Young inequality, for any $\varepsilon>0$, there exists a constant $C_{\varepsilon}>0$ such that
	\begin{equation}\label{eq:FY}
		ab\leqslant \varepsilon a^2+C_{\varepsilon}b^2,\quad\forall a,b\geqslant0.
	\end{equation}

\begin{proof}[Proof of Lemma \ref{regularity_main}]
 Suppose that $\xi_i$ is a minimal curve for $h_L(t,x,y_i,u_i)$ ($i=1,2$). That is
\begin{align*}
	h_L(t,x,y_i,u_i)=u_i+\int^t_0L(\xi_i,\dot{\xi}_i,u_{\xi_i})\ ds,\quad i=1,2,
\end{align*}
where $u_{\xi_i}$ is determined by \eqref{eq:Caratheodory_ODE} with initial condition $u_{\xi_i}(0)=u_i$ ($i=1,2$). We denote $p_i=L_v(\xi_i,\dot{\xi}_i,u_{\xi_i})$ the dual arc of $\xi_i$ ($i=1,2$). 

Moreover, $(\xi_i,p_i,u_{\xi_i})$ satisfies Lie equations
\begin{align*}
	\begin{cases}
\dot{\xi_i}=H_p(\xi_i,p_i,u_{\xi_i})\\
\dot{p}_i=-H_x(\xi_i,p_i,u_{\xi_i})-H_u(\xi_i,p_i,u_{\xi_i})p_i\\
\dot{u}_{\xi_i}=p_i\cdot\dot{\xi}_i-H(\xi_i,p_i,u_{\xi_i})
\end{cases}
\quad \text{on}\ [0,t]
\end{align*}
with
$$
\xi_i(0)=x,\quad\xi_i(t)=y_i,\quad u_{\xi_i}(0)=u_i.
$$
Therefore
\begin{align*}
	\frac 12\frac d{ds}|\xi_2-\xi_1|^2=\langle H_p(\xi_2,p_2,u_{\xi_2})-H_p(\xi_1,p_1,u_{\xi_1}),\xi_2-\xi_1\rangle.
\end{align*}
Integrating over $[s,t]\subset[0,1]$, we conclude that
\begin{equation}\label{eq:regu1}
	\begin{split}
		&\,|\xi_2(t)-\xi_1(t)|^2-|\xi_2(s)-\xi_1(s)|^2\\
		\geqslant&\,-C_1\int^t_s(|\xi_2-\xi_1|^2+|u_{\xi_2}-u_{\xi_1}||\xi_2-\xi_2|+|p_2-p_1||\xi_2-\xi_1|)\ d\tau\\
		\geqslant&\,-C_2\int^t_s|\xi_2-\xi_1|^2d\tau-C_2\int^t_s|u_{\xi_2}-u_{\xi_1}|^2d\tau-C_2\int^t_s|p_2-p_1|^2d\tau.
	\end{split}
\end{equation}
By the Carath\'eodory equation \eqref{eq:Caratheodory_ODE}, it is not difficult to see
\begin{align*}
	|u_{\xi_2}(s)-u_{\xi_1}(s)|\leqslant&\,|u_2-u_1|+C_3\left(\int^s_0|\xi_2-\xi_1|d\tau+\int^s_0|p_2-p_1|d\tau\right)\\
	&+C_4\int^s_0|u_{\xi_2}-u_{\xi_1}|d\tau.
\end{align*}
We denote 
\begin{align*}
	\|u_{\xi_2}-u_{\xi_1}\|=\sup_{s\in[0,t]}|u_{\xi_2}(s)-u_{\xi_1}(s)|,\quad \|\xi_2-\xi_1\|=\sup_{s\in[0,t]}|\xi_2(s)-\xi_1(s)|.
\end{align*}
Invoking Gronwall inequality, we have
\begin{align*}
	\|u_{\xi_2}-u_{\xi_1}\|\leqslant \left(|u_2-u_1|+C_3t\|\xi_2-\xi_1\|+C_3\int^t_0|p_2-p_1|d\tau\right)\exp(C_4t).
\end{align*}
And Jensen inequality implies that
\begin{equation}\label{eq:u^2}
	\|u_{\xi_2}-u_{\xi_1}\|^2\leqslant C_5\left(|u_2-u_1|^2+t\|\xi_2-\xi_1\|^2+t\int^t_0|p_2-p_1|^2d\tau\right),
\end{equation}
since $0<t\leqslant 1$. Now \eqref{eq:regu1} becomes
\begin{equation}\label{eq:basic1}
	\begin{split}
		&\,\|\xi_2-\xi_1\|^2\\
		\leqslant&\,|y_2-y_1|^2+C_2t\|\xi_2-\xi_1\|^2+C_2\int^t_0|p_2-p_1|^2d\tau+C_2t\|u_{\xi_2}-u_{\xi_1}\|^2\\
		\leqslant&\,|y_2-y_1|^2+C_6|u_2-u_1|^2+C_6t\|\xi_2-\xi_1\|^2+C_6\int^t_0|p_2-p_1|^2d\tau.\\
	\end{split}
\end{equation}
Now, 
\begin{align}\label{eq:p}
\frac{d}{dt}\langle p_2-p_1,\xi_2-\xi_1\rangle=I_1-I_2-I_3,
\end{align}
where
\begin{align*}
I_1&=\langle p_2-p_1,H_p(\xi_2,p_2,u_{\xi_2})-H_p(\xi_1,p_1,u_{\xi_1})\rangle,\\
I_2&=\langle H_x(\xi_2,p_2,u_{\xi_2})-H_x(\xi_1,p_1,u_{\xi_1}),\xi_2-\xi_1\rangle,\\
I_3&=\langle H_u(\xi_2,p_2,u_{\xi_2})p_2-H_u(\xi_1,p_1,u_{\xi_1})p_1,\xi_2-\xi_1\rangle.
\end{align*}
Then, we have
\begin{align*}
I_1&=\langle p_2-p_1,\widehat{H}_{px}(\xi_2-\xi_1)\rangle+\langle p_2-p_1,\widehat{H}_{pu}(u_{\xi_2}-u_{\xi_1})\rangle+\langle p_2-p_1,\widehat{H}_{pp}(p_2-p_1)\rangle\\
&\geqslant \nu |p_2-p_1|^2-C_7(|u_{\xi_2}-u_{\xi_1}|\cdot |\xi_2-\xi_1|+|p_2-p_1|\cdot |\xi_2-\xi_1|),
\end{align*}
where $\nu$ is constant such that $|\widehat{H}_{pp}|\geqslant\nu$,
\begin{align*}
	\widehat{H}_{px}=\int^1_0H_{px}(\lambda\xi_2+(1-\lambda)\xi_1,\lambda\dot{\xi}_2+(1-\lambda)\dot{\xi}_1,\lambda u_{\xi_2}+(1-\lambda)u_{\xi_1})d\lambda,
\end{align*}
and $\widehat{H}_{pp}$, $\widehat{H}_{xx}$, $\widehat{H}_{xu}$, $\widehat{H}_{pu}$ and $\widehat{H}_{uu}$ are defined in a similar way. For $I_2,I_3$, we have
\begin{align*}
I_2=&\langle \widehat{H}_{xx}(\xi_2-\xi_1),\xi_2-\xi_1\rangle+\langle \widehat{H}_{xu}(u_{\xi_2}-u_{\xi_1}),\xi_2-\xi_1\rangle\\
&+\langle \widehat{H}_{xp}(p_2-p_1),\xi_2-\xi_1\rangle\\
\leqslant &C_8(|\xi_2-\xi_1|^2+|u_{\xi_2}-u_{\xi_1}|\cdot |\xi_2-\xi_1|+|p_2-p_1|\cdot |\xi_2-\xi_1|),
\end{align*}
and
\begin{align*}
I_3=&\langle (H_u(\xi_2,p_2,u_{\xi_2})-H_u(\xi_1,p_1,u_{\xi_1}))p_2,\xi_2-\xi_1\rangle\\
&+\langle H_u(\xi_1,p_1,u_{\xi_1})(p_2-p_1),\xi_2-\xi_1\rangle\\
=&\langle \widehat{H}_{ux}(\xi_2-\xi_1)p_2,\xi_2-\xi_1\rangle+\langle \widehat{H}_{uu}(u_{\xi_2}-u_{\xi_1})p_2,\xi_2-\xi_1\rangle\\
&+\langle\widehat{H}_{up}(p_2-p_1)p_2,\xi_2-\xi_1\rangle+\langle H_u(\xi_1,p_1,u_{\xi_1})(p_2-p_1),\xi_2-\xi_1\rangle\\
\leqslant &C_9(|\xi_2-\xi_1|^2+|u_{\xi_2}-u_{\xi_1}|\cdot |\xi_2-\xi_1|+|p_2-p_1|\cdot |\xi_2-\xi_1|).
\end{align*}
where $C_9$ is the upper bound for $D^2H$ and $DH$.
Therefore, integrating (\ref{eq:p}) over $[0,t]$, we obtain
\begin{align*}
&\langle p_2(t)-p_1(t),\xi_2(t)-\xi_1(t)\rangle\geqslant\nu\int_0^t|p_2-p_1|^2ds\\
&-C_{10}\int_0^t(|\xi_2-\xi_1|^2+|u_{\xi_2}-u_{\xi_1}|\cdot |\xi_2-\xi_1|+|p_2-p_1|\cdot |\xi_2-\xi_1|)ds\\
&\geqslant(\nu-C_{10}\epsilon)\int_0^t|p_2-p_1|^2ds-C_1(\epsilon)\int_0^t|\xi_2-\xi_1|^2ds-C_{11}\int_0^t|u_{\xi_2}-u_{\xi_1}|^2ds.
\end{align*}
Thus, by (\ref{eq:u^2}), we conclude
\begin{align*}
&\langle p_2(t)-p_1(t),\xi_2(t)-\xi_1(t)\rangle\geqslant(\nu-C_{10}\epsilon-C_5C_{11}t)\int_0^t|p_2-p_1|^2ds\\
&-(C_1(\epsilon)+C_5C_{11}t)\int_0^t|\xi_2-\xi_1|^2ds-C_5C_{11}|u_2-u_1|^2.
\end{align*}
Set $\epsilon=\frac{\nu}{2C_{10}}$. If $t<\frac{\nu}{4C_5C_{11}}$, then
\begin{equation}\label{eq:p_2}
\int_0^t|p_2-p_1|^2ds\leqslant\frac{4}{\nu}\left(C_{12}t\Vert \xi_2-\xi_1\Vert^2+\frac{C_{12}}{t}|y_2-y_1|^2+C_{12}|u_2-u_1|^2\right).
\end{equation}
From (\ref{eq:basic1}), it follows that
\begin{align*}
	&\left(1-C_6t-\frac{4C_{12}C_6}{\nu}t\right)\|\xi_2-\xi_1\|^2\\
	\leqslant&\,\left(1+\frac{4C_{12}C_6}{\nu t}\right)|y_2-y_1|^2+\left(C_6+\frac{4C_{12}C_6}{\nu}\right)|u_2-u_1|^2.
\end{align*}
For $t<\frac{\nu}{2C_6(\nu+4C_{12})}$, we conclude that
\begin{align}\label{eq:xi_2}
\|\xi_1-\xi_2\|^2\leqslant& 2\left(1+\frac{4C_{12}C_6}{\nu t}\right)|y_2-y_1|^2+2\left(C_6+\frac{4C_{12}C_6}{\nu}\right)|u_2-u_1|^2\\
\leqslant& \frac{C_{13}}{t}|y_2-y_1|^2+C_{14}|u_2-u_1|^2.\nonumber
\end{align}
From (\ref{eq:p_2}) and (\ref{eq:xi_2}), one has
\begin{equation}\label{eq:p2}
	\int^t_0|p_2-p_1|^2d\tau\leqslant\frac {C_{15}}t(|y_2-y_1|^2)+C_{16}|u_2-u_1|^2.
\end{equation}
Finally, the combination of \eqref{eq:u^2}, \eqref{eq:xi_2} and \eqref{eq:p2} implies that
\begin{align}\label{eq:uu}
	\|u_{\xi_2}-u_{\xi_1}\|^2\leqslant C_{17}(|y_2-y_1|^2+|u_2-u_1|^2).
\end{align}
Finally, by \eqref{eq:xi_2}, \eqref{eq:p2} and \eqref{eq:uu}, we obtain
\begin{align*}
&\int_0^t|\dot{\xi_2}-\dot{\xi_1}|^2ds=\int_0^t|H_p(\xi_2,\dot{\xi}_2,u_{\xi_2})-H_p(\xi_1,\dot{\xi}_1,u_{\xi_1})|^2ds\\
=&\int_0^t|\widehat{H}_{px}(\xi_2-\xi_1)+\widehat{H}_{pu}(u_{\xi_2}-u_{\xi_1})+\widehat{H}_{pp}(p_2-p_1)|^2ds\\
\leqslant & C_{18}\left(t\|\xi_2-\xi_1\|^2+t\|u_{\xi_2}-u_{\xi_1}\|^2+\int^t_0|p_2-p_1|^2ds\right)\\
\leqslant&\frac {C_{19}}t|y_2-y_1|^2+C_{20}|u_2-u_1|^2.
\end{align*}
This completes our proof.
\end{proof}

\begin{proof}[Proof of Lemma \ref{lemma}]
For any $\tau\in[0,t]$,
\begin{align*}
|\eta_+(\tau)-\eta_-(\tau)|=&\left|\xi_+\left(\frac{t+h}{t}\tau\right)-\xi_-\left(\frac{t-h}{t}\tau\right)\right|\\
\leqslant &\left|\xi_+\left(\frac{t+h}{t}\tau\right)-\xi_+\left(\frac{t-h}{t}\tau\right)\right|+\left|\xi_+\left(\frac{t-h}{t}\tau\right)-\xi_-\left(\frac{t-h}{t}\tau\right)\right|\\
\leqslant &C_1|h|+\left|\xi_+\left(\frac{t-h}{t}\tau\right)-\xi_-\left(\frac{t-h}{t}\tau\right)\right|\\
\leqslant &C_1|h|+\max_{s\in [0,t-h]}|\xi_+(s)-\xi_-(s)|.
\end{align*}
Since
\begin{align*}
|\xi_+(t-h)-\xi_-(t-h)|\leqslant &|\xi_+(t-h)-\xi_+(t+h)|+|\xi_+(t+h)-\xi_-(t-h)|\\
\leqslant & 2(C_2|h|+|z|),
\end{align*}
by (\ref{hu:xi}), we obtain
\begin{align*}
	\max_{s\in [0,t-h]}|\xi_+(s)-\xi_-(s)|^2\leqslant \frac{C_3}t(|h|^2+|z|^2)+C_3|k|^2.
\end{align*}
Thus, (\ref{hu:eta}) follows immediately.

Similarly, we have
\begin{align*}
|\dot{\eta}_+-\dot{\eta}_-|=&\left|\frac{t+h}{t}\dot{\xi}_+\left(\frac{t+h}{t}\tau\right)-\frac{t-h}{t}\dot{\xi}_-\left(\frac{t-h}{t}\tau\right)\right|\\
\leqslant &\left|\frac{t+h}{t}\dot{\xi}_+\left(\frac{t+h}{t}\tau\right)-\frac{t+h}{t}\dot{\xi}_-\left(\frac{t-h}{t}\tau\right)\right|\\
&+\left|\frac{t+h}{t}\dot{\xi}_-\left(\frac{t-h}{t}\tau\right)-\frac{t-h}{t}\dot{\xi}_-\left(\frac{t-h}{t}\tau\right)\right|\\
\leqslant &\frac{t+h}{t}\left|\dot{\xi}_+\left(\frac{t+h}{t}\tau\right)-\dot{\xi}_-\left(\frac{t-h}{t}\tau\right)\right|+C_4\frac{|h|}{t}.
\end{align*}
On the other hand,
\begin{align*}
&\left|\dot{\xi}_+\left(\frac{t+h}{t}\tau\right)-\dot{\xi}_-\left(\frac{t-h}{t}\tau\right)\right|\\
\leqslant &\left|\dot{\xi}_+\left(\frac{t+h}{t}\tau\right)-\dot{\xi}_+\left(\frac{t-h}{t}\tau\right)\right|+\left|\dot{\xi}_+\left(\frac{t-h}{t}\tau\right)-\dot{\xi}_-\left(\frac{t-h}{t}\tau\right)\right|\\
\leqslant &C_5\frac{|h|}{t}+\left|\dot{\xi}_+\left(\frac{t-h}{t}\tau\right)-\dot{\xi}_-\left(\frac{t-h}{t}\tau\right)\right|.
\end{align*}
Thus, we conclude that
\begin{align*}
\int_0^t|\dot{\eta}_+-\dot{\eta}_-|^2d\tau\leqslant&C_6\frac{|h|^2}{t}+\frac{(t+h)^2}{t^2}\int_0^t\left|\dot{\xi}_+\left(\frac{t-h}{t}\tau\right)-\dot{\xi}_-\left(\frac{t-h}{t}\tau\right)\right|^2d\tau\\
=&C_6\frac{|h|^2}{t}+\frac{(t+h)^2}{t^2}\cdot\frac{t}{t-h}\int_0^{t-h}|\dot{\xi}_+-\dot{\xi}_-|^2d\tau\\
\leqslant & C_6\frac{|h|^2}{t}+\frac{(t+h)^2}{t^2}\cdot\frac{t}{t-h}\cdot\left(\frac{C_7}{t}(|h|^2+|z|^2)+C_7|k|^2\right).
\end{align*}
This leads to (\ref{hu:doteta}).

From (\ref{hu:eq}), it follows that
\begin{align*}
v_+(\tau)-v_-(\tau)=&\,v_+(0)-v_-(0)+\int^{\tau}_0L\left(\eta_+,\frac{t}{t+h}\dot{\eta}_+,v_+\right)-L\left(\eta_-,\frac{t}{t-h}\dot{\eta}_-,v_-\right)ds\\
		&\,+\frac{h}t\int^{\tau}_0L\left(\eta_+,\frac{t}{t+h}\dot{\eta}_+,v_+\right)+L\left(\eta_-,\frac{t}{t-h}\dot{\eta}_-,v_-\right)ds\\
		\leqslant&\,2|k|+I_1+C_8|h|,
	\end{align*}
where
\begin{equation}\label{hu:I_1}
I_1=\int^{\tau}_0L\left(\eta_+,v_+,\frac{t}{t+h}\dot{\eta}_+\right)-L\left(\eta_-,v_-,\frac{t}{t-h}\dot{\eta}_-\right)ds.
\end{equation}	
On the other hand,
\begin{align}
\left|\frac{t}{t+h}\dot{\eta}_+-\frac{t}{t-h}\dot{\eta}_-\right|\leqslant &\left|\frac{t}{t+h}\dot{\eta}_+-\frac{t}{t+h}\dot{\eta}_-\right|+\left|\frac{t}{t+h}\dot{\eta}_--\frac{t}{t-h}\dot{\eta}_-\right|\label{hu:tth}\\
\leqslant & 2|\dot{\eta}_+-\dot{\eta}_-|+C_9\frac{|h|}{t}.\nonumber
\end{align}
Thus,
\begin{equation}\label{hu:I}
I_1\leqslant C_{10}(t\|\eta_+-\eta_+\|+\int_0^\tau |v_+-v_-|ds+\int_0^\tau|\dot{\eta}_+-\dot{\eta}_-|ds+|h|),
\end{equation}
which yields that
$$|v_+-v_-|\leqslant 2|k|+C_8|h|+C_{10}(t\|\eta_+-\eta_+\|+\int_0^\tau |v_+-v_-|ds+\int_0^\tau|\dot{\eta}_+-\dot{\eta}_-|ds+|h|).$$
Invoking Gronwall inequality, we have
\begin{align*}
&\|v_+-v_-\|\\
\leqslant &\left(2|k|+(C_8+C_{10})|h|+C_{10}t\|\eta_+-\eta_+\|+C_{10}\int_0^t|\dot{\eta}_+-\dot{\eta}_-|ds\right)\exp(C_{10}t).
\end{align*}
By (\ref{hu:eta}) and (\ref{hu:doteta}), we have
\begin{align*}
\|v_+-v_-\|^2\leqslant &C_{11}(|h|^2+|k|^2+t\|\eta_+-\eta_+\|^2+t\int_0^t|\dot{\eta}_+-\dot{\eta}_-|^2ds)\\
\leqslant &C_{12}(|h|^2+|z|^2+|k|^2).
\end{align*}

Finally, we are ready to prove (\ref{hu:vv}). By the definitions of $\eta^{\pm}$ and $v^{\pm}$, we obtain
	\begin{align*}
		&\,v_+(\tau)+v_-(\tau)-2u_{\xi}(\tau)\\
		=&\,\int^{\tau}_0\left\{\frac{t+h}tL\left(\eta_+,\frac t{t+h}\dot{\eta}_+,v_+\right)+\frac{t-h}tL\left(\eta_-,\frac t{t-h}\dot{\eta}_-,v_-\right)-2L(\xi,\dot{\xi},u_{\xi})\right\} ds\\
		=&\, I_2+\frac{h}{t}I_1,
		\end{align*}
	where $I_1$ is defined by (\ref{hu:I_1}) and
\begin{align*}
	I_2=&\,\int^{\tau}_0\left\{L\left(\eta_+,\frac t{t+h}\dot{\eta}_+,v_+\right)+L\left(\eta_-,\frac t{t-h}\dot{\eta}_-,v_-\right)-2L(\xi,\dot{\xi},u_{\xi})\right\} ds.
	\end{align*}
Due to (\ref{hu:eta})-(\ref{hu:v}) and \eqref{hu:I}, we have that
	\begin{align*}
		\frac{h}{t}I_1\leqslant&\,\frac {C_{13}}t(|h|^2+|z|^2)+C_{13}|k|^2.
			\end{align*}
	Since $L$ is of $C^2$ class, then by (\ref{semiconcave}) we have that
	\begin{equation}\label{eq:semiconcavity1}
		\begin{split}
			I_2\leqslant&\,\int^{\tau}_0\left|L\left(\eta_+,\frac t{t+h}\dot{\eta}_+,v_+\right)+L\left(\eta_-,\frac t{t-h}\dot{\eta}_-v_-\right)-2L(\xi,\dot{\xi},u_{\xi})\right| ds\\	
			\leqslant&\,C_{14}\cdot\int^t_0\left(|\eta_+-\eta_-|^2+|v_--v_+|^2+\left|\frac t{t+h}\dot{\eta}_+-\frac t{t-h}\dot{\eta}_-\right|^2\right)\ ds\\	
		&\,+C_{14}\cdot\int^t_0(|\eta_++\eta_--2\xi|+|v_-+v_+-2u_{\xi}|)\ ds\\
		&\,+C_{14}\cdot\int^t_0\left|\frac t{t+h}\dot{\eta}_++\frac t{t-h}\dot{\eta}_--2\dot{\xi}\right|\ ds.
		\end{split}
	\end{equation}
	Since 
	$$\eta_++\eta_--2\xi=0,$$
	\begin{align*}
	&\left|\frac t{t+h}\dot{\eta}_++\frac t{t-h}\dot{\eta}_--2\dot{\xi}\right|=\left|\frac h{t+h}\dot{\eta}_+-\frac h{t-h}\dot{\eta}_-\right|\\
	\leqslant& \left|\frac h{t+h}(\dot{\eta}_+-\dot{\eta}_-)\right|+\left|\left(\frac h{t+h}-\frac h{t-h}\right)\dot{\eta}_-\right|\leqslant C_{15}\left(\frac{|h|}{t}\cdot|\dot{\eta}_+-\dot{\eta}_-|+\frac{|h|^2}{t^2}\right)\\
	\leqslant&\,C_{16}\left(|\dot{\eta}_+-\dot{\eta}_-|^2+\frac {|h|^2}{t^2}\right),
	\end{align*}
by \eqref{hu:eta}-\eqref{hu:v} and \eqref{hu:tth} we obtain
\begin{align*}
	|v^-+v^+-2u_{\xi}|\leqslant\left\{\frac {C_{17}}t(|h|^2+|z|^2)+C_{17}|k|^2)\right\}\exp(C_{14}t).
\end{align*}
Then, (\ref{hu:vv}) follows immediately. This completes the proof.
\end{proof}

\begin{proof}[Proof of Theorem \ref{hu:uniform covex}]
The proof is based on the estimate
\begin{align*}
	I\geqslant&\,\int^{t+h}_0L(\xi_+,\dot{\xi}_+,u_+)\ ds+\int^{t-h}_0L(\xi_-,\dot{\xi}_-,u_-)\ ds-2\int^{t}_0L(\xi,\dot{\xi},u_{\xi})\ ds\\
	=&\,\frac{t+h}t\int^t_0L\left(\eta_+,\frac{t}{t+h}\dot{\eta}^+,v_+\right)d\tau+\frac{t-h}t\int^t_0L\left(\eta_-,\frac{t}{t-h}\dot{\eta}^-,v_-\right)d\tau\\
	&\,-2\int^{t}_0L(\xi,\dot{\xi},u_{\xi})\ d\tau.
\end{align*}
Set
\begin{align*}
	I_1=&\,\frac ht\int^t_0\left\{L\left(\eta_+,\frac{t}{t+h}\dot{\eta}_+,v_+\right)-L\left(\eta_-,\frac{t}{t-h}\dot{\eta}_-,v_-\right)\right\}d\tau,\\
	I_2=&\,\int^t_0\left\{L\left(\eta_+,\dot{\eta}_+,v_+\right)+L\left(\eta_-,\dot{\eta}_-,v_-\right)-2L(\xi,\dot{\xi},u_{\xi})\right\}d\tau,\\
	I_3=&\,\int^t_0L\left(\eta_+,\frac{t}{t+h}\dot{\eta}_+,v_+\right)d\tau-\int_0^tL(\eta_+,\dot{\eta}_+,v_+)d\tau\\
	&+\,\int^t_0L\left(\eta_-,\frac{t}{t-h}\dot{\eta}_-,v_-\right)d\tau-\int_0^tL\left(\eta_-,\dot{\eta}_-,v_-\right)d\tau.
\end{align*}
Then $I\geqslant I_1+I_2+I_3$.

\medskip

\noindent\textbf{Estimate of $I_1$:} By (\ref{hu:eta})-(\ref{hu:v}) and (\ref{hu:I}), we have
\begin{align*}
I_1\geqslant& -C_1\frac{|h|}{t}\cdot\left(t\|\eta_+-\eta_+\|+\int_0^t |v_+-v_-|ds+\int_0^t|\dot{\eta}_+-\dot{\eta}_-|ds+|h|\right)\\
\geqslant & -\frac{C_2}{t}(|h|^2+|z|^2)+C_2|k|^2.
\end{align*}

\medskip
\noindent\textbf{Estimate of $I_2$:} 
\begin{align*}
I_2=&\int^t_0\left\{L\left(\eta_+,\dot{\eta}_+,v_+\right)+L\left(\eta_-,\dot{\eta}_-,v_-\right)-2L\left(\xi,\dot{\xi},\frac{v_++v_-}{2}\right)\right\}d\tau\\
&+2\int^t_0L\left(\xi,\dot{\xi},\frac{v_++v_-}{2}\right)-L\left(\xi,\dot{\xi},u_\xi\right)d\tau:=I_4+I_5.
\end{align*}
For $I_4$, we have
\begin{align*}
I_4=&\,\int^t_0\left\{L\left(\xi,\dot{\eta}_+,\frac{v_++v_-}{2}\right)+L\left(\xi,\dot{\eta}_-,\frac{v_++v_-}{2}\right)-2L\left(\xi,\dot{\xi},\frac{v_++v_-}{2}\right)\right\}d\tau\\
&+\int^t_0L\left(\eta_+,\dot{\eta}_+,v_+\right)d\tau-\int^t_0L\left(\xi,\dot{\eta}_+,\frac{v_++v_-}{2}\right)d\tau\\
&+\int^t_0L\left(\eta_-,\dot{\eta}_-,v_-\right)d\tau-\int^t_0L\left(\xi,\dot{\eta}_-,\frac{v_++v_-}{2}\right)d\tau\\
\geqslant &\nu\int_0^t\left|\frac{\dot{\eta}_+-\dot{\eta}_-}{2}\right|^2d\tau+ \int_0^t\int_0^1\left\langle \tilde{L}_x,\frac{\eta_+-\eta_-}{2}\right\rangle+\left\langle\tilde{L}_u,\frac{v_+-v_-}{2}\right\rangle d\lambda d\tau,
\end{align*}
where
\begin{align*}
\tilde{L}_x=&L_x\left(\lambda\eta_++(1-\lambda)\xi,\dot{\eta}_+,\lambda v_++(1-\lambda)\frac{v_++v_-}{2}\right)\\
&-L_x\left(\lambda\eta_-+(1-\lambda)\xi,\dot{\eta}_-,\lambda v_-+(1-\lambda)\frac{v_++v_-}{2}\right),\\
\tilde{L}_u=&L_u\left(\lambda\eta_++(1-\lambda)\xi,\dot{\eta}_+,\lambda v_++(1-\lambda)\frac{v_++v_-}{2}\right)\\
&-L_u\left(\lambda\eta_-+(1-\lambda)\xi,\dot{\eta}_-,\lambda v_-+(1-\lambda)\frac{v_++v_-}{2}\right).
\end{align*}
It follows that
\begin{align*}
I_2\geqslant &\nu\int_0^t\left|\frac{\dot{\eta}_+-\dot{\eta}_-}{2}\right|^2d\tau-C_4\int_0^t(|v_++v_--2u_\xi|+\left|\eta_+-\eta_-\right|^2\\
&+|v_+-v_-|^2+|\dot{\eta}_+-\dot{\eta}_-|\cdot \left|\eta_+-\eta_-\right|+|\dot{\eta}_+-\dot{\eta}_-|\cdot |v_+-v_-|)d\tau.
\end{align*}

\medskip
\noindent\textbf{Estimate of $I_3$:} 
\begin{align*}
I_3\geqslant &\int_0^t \left\{\langle L_v(\eta_+,\dot{\eta}_+,v_+) ,-\frac{h}{t+h}\dot{\eta}_+\rangle +\frac{\nu h^2}{|t+h|^2}|\dot{\eta}_+|^2\right\}d\tau\\
&+\int_0^t\left\{\langle L_v(\eta_-,\dot{\eta}_-,v_-),\frac{h}{t-h}\dot{\eta}_-\rangle +\frac{\nu h^2}{|t-h|^2}|\dot{\eta}_-|^2\right\}d\tau\\
=&I_6+I_7,
\end{align*}
where
$$I_6=\int_0^t\langle L_v(\eta_+,\dot{\eta}_+,v_+) ,-\frac{h}{t+h}\dot{\eta}_+\rangle+\langle L_v(\eta_-,\dot{\eta}_-,v_-) ,\frac{h}{t-h}\dot{\eta}_-\rangle d\tau,$$
and 
$$I_7=\int_0^t\frac{\nu h^2}{|t+h|^2}|\dot{\eta}_+|^2+\frac{\nu h^2}{|t-h|^2}|\dot{\eta}_-|^2d\tau.$$
For $I_6$, we have
\begin{align*}
I_6=&\frac{h}{t-h}\int_0^t\langle L_v(\eta_-,\dot{\eta}_-,v_-) ,\dot{\eta}_--\dot{\eta}_+\rangle d\tau\\
&+\left(\frac{h}{t-h}-\frac{h}{t+h}\right)\int_0^t\langle L_v(\eta_-,\dot{\eta}_-,v_-),\dot{\eta}_+\rangle d\tau\\
&+\frac{h}{t+h}\int_0^t\langle L_v(\eta_-,\dot{\eta}_-,v_-)-L_v(\eta_+,\dot{\eta}_+,v_+),\dot{\eta}\rangle d\tau.
\end{align*}
Then, by (\ref{hu:eta})-(\ref{hu:v}) we have
\begin{align*}
I_6\geqslant & -C_5 \frac{2|h|}{t}\int_0^t |\dot{\eta}_+-\dot{\eta}_-|d\tau-C_5\frac{2|h|^2t}{t^2-h^2}\\
&-C_5\frac{|h|}{t+h}\left\{\int_0^t|\dot{\eta}_+-\dot{\eta}_-|d\tau+\int_0^t|v_+-v_-|d\tau+\int_0^t|\eta_+-\eta_-|d\tau \right\}\\
\geqslant &-\frac{C_6}{t}(|k|^2+|h|^2)-C_6|z|^2.
\end{align*}
For $I_7$, we obtain
\begin{align*}
I_7=&\int_0^{t+h}\frac{\nu h^2}{|t+h|^2}|\dot{\xi}_+|^2\cdot \frac{t}{t+h}ds+\int_0^{t-h}\frac{\nu h^2}{|t-h|^2}|\dot{\xi}_-|^2\cdot \frac{t}{t-h}ds\\
\geqslant&\frac{\nu h^2 t}{|t+h|^2}\inf_{s\in [0,t+h]}|\dot{\xi}_+|^2+\frac{\nu h^2 t}{|t-h|^2}\inf_{s\in [0,t-h]}|\dot{\xi}_-|^2\\
=&\frac{4\nu h^2 }{9t}(\inf_{s\in [0,t+h]}|\dot{\xi}_+|^2+\inf_{s\in [0,t-h]}|\dot{\xi}_-|^2).
\end{align*}
Now, we are ready to prove (\ref{convex}). Combining the estimates of $I_1$, $I_2$ and $I_3$, we obtain, thanks to Lemma \ref{lemma}
\begin{equation}\label{hu:c}
	\begin{split}
		&\,h_L(t+h,x,y+z,u+k)+h_L(t-h,x,y-z,u-k)-2h_L(t,x,y,u)\\
	\geqslant &-\frac{C_2+C_6}{t}(|h|^2+|z|^2)-(C_2+C_6)|k|^2-C_4\int_0^t\Big(|v_++v_--2u_\xi|+\left|\eta_+-\eta_-\right|^2\\
	&+|v_+-v_-|^2+|\dot{\eta}_+-\dot{\eta}_-|\cdot \left|\eta_+-\eta_-\right|+|\dot{\eta}_+-\dot{\eta}_-|\cdot |v_+-v_-|\Big)d\tau\\
	\geqslant&\,-\frac{C_{7}}{t}(|h|^2+|z|^2)-C_7|k|^2.
	\end{split}
\end{equation}
This completes the semiconvexity estimate \eqref{convex}.

Finally, in order to prove (\ref{uniform convex}), we take $h=0, k=0$. Then we conclude $I_1,\,I_3\geqslant 0$. Therefore, for any $\epsilon>0$, the estimate of $I_2$ yields that
\begin{align*}
h_L(&\,t,x,y+z,u)+h_L(t,x,y-z,u)-2h_L(t,x,y,u)\\
\geqslant &\frac{\nu}{4}\int_0^t\left|\dot{\eta}_+-\dot{\eta}_-\right|^2d\tau-C_4\int_0^t(|v_++v_--2u_\xi|+\left|\eta_+-\eta_-\right|^2\\
&+|v_+-v_-|^2+|\dot{\eta}_+-\dot{\eta}_-|\cdot \left|\eta_+-\eta_-\right|+|\dot{\eta}_+-\dot{\eta}_-|\cdot |v_+-v_-|)d\tau\\
\geqslant & \left(\frac{\nu}{4}-\epsilon\right)\int_0^t\left|\dot{\eta}_+-\dot{\eta}_-\right|^2d\tau-C_4\int_0^t(|v_++v_--2u_\xi|d\tau\\
&-\left(C_4+\frac{C_4^2}{2\epsilon}\right)\int_0^t(\left|\eta_+-\eta_-\right|^2+|v_+-v_-|^2)d\tau.
\end{align*}
Now, since $\dot{\eta}_+-\dot{\eta}_-$ is an arc connecting $0$ to $z$, comparison with $s\rightarrow \frac{s}{t}z$ yields
$$\int_0^t\left|\dot{\eta}_+-\dot{\eta}_-\right|^2d\tau\geqslant \frac{|z|^2}{t}.$$
Hence, taking $\epsilon=\nu/8$, and appealing to Lemma \ref{lemma}, it follows that
\begin{align*}
&h_L(\,t,x,y+z,u)+h_L(t,x,y-z,u)-2h_L(t,x,y,u)\\
\geqslant &\left\{\frac{\nu}{8t}-\left(C_4+\frac{C_4^2}{2\epsilon}\right)C\right\}|z|^2,
\end{align*}
where $C>0$ is determined by Lemma \ref{lemma}. Now choosing $t_{\lambda,\delta}$ such that
$$\frac{\nu}{8t_{\lambda,\delta}}-\left(C_4+\frac{C_4^2}{2\epsilon}\right)C>0,$$
one completes the proof.
\end{proof}

\begin{proof}[Proof of Corollary \ref{cor:convexity_t}]
By the estimate of $I_1,I_2,I_3$ in Theorem \ref{hu:uniform covex} and take $z=0,\,k=0$, we have
\begin{align*}
&h_L(t+h,x,y,u)+h_L(t-h,x,y,u)-2h_L(t,x,y,u)\\
\geqslant &\frac{4\nu |h|^2 }{9t}(\inf_{s\in [0,t+h]}|\dot{\xi}_+|^2+\inf_{s\in [0,t-h]}|\dot{\xi}_-|^2)-\frac{C_7}{t}|h|^2\\
\geqslant &\frac{C}{t}|h|^2,
\end{align*}
provided that
$$\inf_{s\in [0,t+h]}|\dot{\xi}_+|^2+\inf_{s\in [0,t-h]}|\dot{\xi}_-|^2> \frac{9 C_7}{4\nu},$$
where $C_7$ is defined by (\ref{hu:c}).
\end{proof}

\bibliographystyle{plain}
\bibliography{mybib}
\end{document}